\documentclass[]{interact}

\usepackage{epstopdf}
\usepackage[caption=false]{subfig}

\usepackage{amscd,multirow}
\usepackage{graphicx}
\usepackage{hyperref}
\hypersetup{hypertex=true, 
	colorlinks=true, 
	linkcolor=blue, 
	anchorcolor=blue, 
	citecolor=blue}
\usepackage{subeqnarray,cases}

\usepackage[numbers,sort&compress]{natbib}
\bibpunct[, ]{[}{]}{,}{n}{,}{,}
\renewcommand\bibfont{\fontsize{10}{12}\selectfont}
\makeatletter
\def\NAT@def@citea{\def\@citea{\NAT@separator}}
\makeatother

\theoremstyle{plain}
\newtheorem{theorem}{Theorem}[section]
\newtheorem{lemma}[theorem]{Lemma}

\theoremstyle{definition}

\theoremstyle{remark}
\newtheorem{remark}{Remark}

\begin{document}


\title{Tikhonov regularized second-order plus first-order  primal-dual dynamical systems with asymptotically vanishing damping for linear equality constrained convex optimization problems}

\author{
\name{Ting-Ting Zhu\textsuperscript{a}, Rong Hu\textsuperscript{b}\thanks{CONTACT Rong Hu. Email: ronghumath@aliyun.com} and Ya-Ping Fang\textsuperscript{a}\thanks{CONTACT Ya-Ping Fang. Email: ypfang@scu.edu.cn}}
\affil{\textsuperscript{a} Department of Mathematics, Sichuan University, Chengdu, Sichuan, P.R. China; \textsuperscript{b}Department of Applied Mathematics, Chengdu University of Information Technology, Chengdu, Sichuan, P.R. China}
}

\maketitle

\begin{abstract}
In this paper, in the setting of Hilbert spaces, we consider a Tikhonov regularized second-order plus first-order primal-dual dynamical system with asymptotically vanishing damping for a linear equality constrained convex optimization problem. It is shown that convergence properties of the proposed dynamical system depend  upon the choice of the Tikhonov regularization parameter. When the Tikhonov regularization parameter decays rapidly to zero, we derive the fast convergence rates of the primal-dual gap, the objective residual, the feasibility violation and the gradient norm of the objective function along the generated trajectory. When the Tikhonov regularization parameter decreases slowly to zero, we prove the strong convergence of the primal trajectory of the Tikhonov regularized dynamical system to the minimal norm solution of the linear equality constrained convex optimization problem. Numerical experiments are performed to illustrate the efficiency of our approach.
\end{abstract}


\amscodename\  34D05; 37N40; 46N10; 90C25.

\section{Introduction}
Let $\mathcal{X}$ and $\mathcal{Y}$ be two real Hilbert spaces which are endowed with the inner product $\langle \cdot, \cdot\rangle$ and the associated norm $\|\cdot\|$. In this paper, we are concerned with the  linear equality constrained convex  optimization problem
\begin{equation}\label{z1}
	\min_{x\in\mathcal{X}}  \quad f(x), \quad \text{ s.t. }  \  Ax = b,
\end{equation}
where $A: \mathcal{X}\rightarrow \mathcal{Y}$ is a continuous linear operator, $b\in\mathcal{Y}$, and $f: \mathcal{X}\rightarrow \mathbb{R}$ is a continuously differentiable convex function such that the gradient operator $\nabla f$ of $f$  is Lipschitz continuous over  $\mathcal{X}$. The problem \eqref{z1} is a basic model for many important applications arising in various areas, such as machine learning \cite{ZLinandLiHandFang(2020)}, image recovery \cite{GoldsteinTandDonoghue(2014)}, the energy dispatch of power grids \cite{PYiandHongYandLiu(2015)} and network optimization \cite{ZengXLandLeiJLandChenJ(2022)}. When $A=0$ and $b=0$, the problem \eqref{z1} collapses to the unconstrained convex optimization problem
\begin{equation}\label{up-f}
	\min_{x\in\mathcal{X}}  \quad f(x).
\end{equation}

In the literature there has been a flurry of research  on second-order dynamical systems for the unconstrained optimization problem \eqref{up-f}. A seminal work is due to Polyak \cite{Polyak(1964)} who proposed the heavy ball with friction system
\begin{eqnarray*}
	\text{(HBF)}\quad\quad 
	\ddot{x}(t)+\gamma \dot{x}(t)+\nabla f(x(t))=0, 
\end{eqnarray*}
where the damping coefficient $\gamma>0$ is a constant. Another seminal work is due to  Su, Boyd, and Cand$\grave{e}$s \cite{SuBoydandCandes(2016)} who proposed an inertial dynamical system with the asymptotically vanishing damping 
\begin{eqnarray*}
	\text{(AVD)}_{\alpha}\quad\quad
	\ddot{x}(t)+\frac{\alpha}{t} \dot{x}(t)+\nabla f(x(t))=0, 
\end{eqnarray*}
where $\alpha>0$. It has been shown by Su, Boyd, and Cand$\grave{e}$s  \cite{SuBoydandCandes(2016)} that $\text{(AVD)}_{\alpha}$ with $\alpha=3$ is  the continuous time limit of Nesterov's accelerated gradient algorithm \cite{Nesterov(1983), Nesterov(2013)}. In the case $\alpha\geq3$, they proved that the objective residual along the trajectory generated by $\text{(AVD)}_{\alpha}$ enjoys the $\mathcal{O}(\frac{1}{t^2})$ convergence rate.  $\text{(AVD)}_{\alpha}$  provides an important model for designing fast methods for the problem \eqref{up-f} since the damping coefficient $\frac{\alpha}{t}$ vanishes in a controlled manner, neither too fast nor too slowly.  For more results on  $\text{(AVD)}_{\alpha}$  and other second-order dynamical systems with time-dependent damping we refer to \cite{Attouch2018MP, Wilson2021, ShiBMP}.

The Tikhonov regularization techniques have played an important role in dynamical system methods, which may assure the strong convergence of the generated trajectory to the minimal norm solution of the problem under consideration. For results on the Tikhonov regularized first-order dynamical systems we refer to  \cite{AlvarezandCabot(2006), Attouch(1996), CominettiPS(2008)}. Recently, Tikhonov regularizations of second-order dynamical systems for unconstrained optimization problems have attached much attention of  researchers.  Attouch and Czarnecki \cite{AttouchandCzarnecki(2002)} considered a Tikhonov regularization  of  Polyak's heavy ball with friction system $\text{(HBF})$ by  proposing  the following dynamical system
\begin{eqnarray*}
	\text{(HBF)}_\epsilon \quad\quad \ddot{x}(t)+\gamma \dot{x}(t)+\nabla f(x(t))+\epsilon(t)x(t)=0,  
\end{eqnarray*}
where $\epsilon:[0,+\infty)\to[0,+\infty)$ is the Tikhonov regularization parameter.  Under the condition $\int_{0}^{+\infty}\epsilon(t)dt=+\infty$,  Attouch and Czarnecki \cite{AttouchandCzarnecki(2002)} showed that the trajectory generated by $\text{(HBF)}_\epsilon$ converges strongly to the minimal norm solution of the problem \eqref{up-f}. See also \cite{AttouchandCzarnecki(2017), Cabot(2004), JendoubiandMay(2010)}. Attouch, Chbani, and Riahi \cite{AttouchZH2018} further considered a Tikhonov regularization of  $\text{(AVD)}_{\alpha}$ by proposing the  dynamical system
\begin{eqnarray*}
	\text{(AVD)}_{\alpha, \epsilon}\quad\quad
	\ddot{x}(t)+\frac{\alpha}{t} \dot{x}(t)+\nabla f(x(t))+\epsilon(t)x(t)=0. 
\end{eqnarray*}
Attouch, Chbani, and Riahi \cite{AttouchZH2018} showed that  the convergence properties of $\text{(AVD)}_{\alpha, \epsilon}$ depend  upon the speed of convergence of the parameter $\epsilon(t)$ to zero: (i)  When $\epsilon(t)$ decays rapidly to zero, $\text{(AVD)}_{\alpha, \epsilon}$ owns the fast convergence properties analogous to $\text{(AVD)}_{\alpha}$. (ii)  When $\epsilon(t)$ decreases slowly to zero, the trajectory of $\text{(AVD)}_{\alpha, \epsilon}$ converges strongly to the  minimal norm solution of the problem \eqref{up-f}. In the case $\epsilon(t)=\frac{1}{t^r}$ with $r>0$, it has been shown that $r=2$ is a critical value which separates two above cases. For more results on the Tikhonov regularized second-order dynamical systems for unconstrained optimization problems, we refer the reader to \cite{AttouchandSzilardLaszlo(2021), Laszlo(2023), AttouchZH2018, XuWen2021}.

Very recently, some researchers were devoted to the study of second-order primal-dual dynamical system for the linear equality constrained convex optimization problem \eqref{z1}. Zeng, Lei, and Chen \cite{ZengXLandLeiJLandChenJ(2022)} proposed and studied the following second-order primal-dual dynamical system
\begin{eqnarray*}
	\text{(Z-AVD)}_{\alpha, \theta}\quad
	\begin{cases}
		\ddot{x}(t)+\frac{\alpha}{t}\dot{x}(t)&=-\nabla f(x(t))-A^T(\lambda(t)+\theta t\dot{\lambda}(t))-A^T(Ax(t)-b),\\
		\ddot{\lambda}(t)+\frac{\alpha}{t}\dot{\lambda}(t)&=A(x(t)+\theta t\dot{x}(t))-b,
	\end{cases}
\end{eqnarray*}
where $t\ge t_0>0$, $\alpha>0$ and $\theta>0$. Zeng, Lei, and Chen \cite{ZengXLandLeiJLandChenJ(2022)} extended the work of Su, Boyd, and Cand$\grave{e}$s \cite{SuBoydandCandes(2016)} from the unconstrained optimization problem \eqref{up-f} to the linear equality constrained optimization problem \eqref{z1}. He, Hu, and Fang \cite{HeHuFangetal(2021)} and Attouch et al. \cite{AttouchADMM(2022)} further proposed and studied second-order primal-dual dynamical systems with time-dependent damping coefficients for a linear equality constrained convex optimization problem with a separable structure.  Fast convergences of primal-dual gap and feasibility violation along the trajectories were derived in \cite{ZengXLandLeiJLandChenJ(2022), HeHuFangetal(2021), AttouchADMM(2022)}, and  fast convergence of the objective residual along the trajectory was established in \cite{ AttouchADMM(2022)}. Besides the fast convergence properties of the primal-dual gap, the objective residual, and the feasibility violation along the trajectories, Bot and Nguyen \cite{BNguyen2022} further proved that the primal-dual trajectory of a second-order primal-dual dynamical system  converges weakly to a primal-dual optimal solution of the linear equality constrained convex optimization problem \eqref{z1}. It is worth mentioning that the second-order dynamical systems considered in \cite{ZengXLandLeiJLandChenJ(2022), HeHuFangetal(2021), AttouchADMM(2022), BNguyen2022} involve second-order terms for both the primal variable and the dual variable. Different from \cite{ZengXLandLeiJLandChenJ(2022), HeHuFangetal(2021),AttouchADMM(2022),BNguyen2022},  He, Hu, and Fang \cite{HeHFetal(2022)}  considered a second-order plus first-order  primal-dual dynamical system with a fixed damping coefficient and a time-dependent scaling coefficient for the problem \eqref{z1}, where the second-order term appears only in the primal variable. Recently, He, Hu, and Fang \cite{HeHFiietal(2022)} further proposed  the following second-order plus first-order primal-dual dynamical system
\begin{eqnarray*}
	\text{(He-AVD)}_{\alpha}\quad
	\begin{cases}
		\ddot{x}(t)+\frac{\alpha}{t}\dot{x}(t)&=-\beta(t)(\nabla f(x(t))+A^T\lambda(t))
		+\varepsilon(t),\\
		\dot{\lambda}(t)&=t\beta(t)(A(x(t)+ \frac{t}{\alpha-1}\dot{x}(t))-b), 
	\end{cases}
\end{eqnarray*}
where $t\geq t_0>0$, $\alpha>1$, $\beta: [t_0, +\infty)\to (0, +\infty)$ is a scaling coefficient, and $\varepsilon: [t_0, +\infty)\to \mathbb{R}^n$ can be viewed as a perturbation term.  He, Hu, and Fang \cite{HeHFiietal(2022)} proved that all of the primal-dual gap, the objective function value and the feasibility violation along the trajectories generated by $\text{(He-AVD)}_{\alpha}$ own the fast convergence properties, which are analogous to the convergence results of second-order plus second-order dynamical systems considered in \cite{ZengXLandLeiJLandChenJ(2022), HeHuFangetal(2021), AttouchADMM(2022), BNguyen2022}.

In this paper,  we consider the following Tikhonov regularized second-order plus first-order primal-dual dynamical system  with asymptotically vanishing damping
\begin{eqnarray}\label{z2}
	\begin{cases}
		\ddot{x}(t)+\frac{\alpha}{t}\dot{x}(t)&=-\nabla f(x(t))-A^T\lambda(t)-\rho A^T(Ax(t)-b)
		-\epsilon(t)x(t),\\
		\dot{\lambda}(t)&=t(A(x(t)+\dfrac{t}{\alpha-1}\dot{x}(t))-b),
	\end{cases}
\end{eqnarray}
where $t\geq t_0>0$, $\alpha>1$, $\rho\ge 0$ is the penalty parameter of the augmented Lagrangian function $\mathcal{L}_{\rho}(x, \lambda)$, and $\epsilon: [t_0, +\infty)\rightarrow [0,+\infty)$  is the Tikhonov regularization parameter which  is a $\mathcal{C}^1$ nonincreasing function satisfying  $\lim_{t\rightarrow+\infty}\epsilon(t)=0$. When $\epsilon(t)=0$ and $\rho=0$, the primal-dual dynamical system \eqref{z2} becomes a particular case of  $\text{(He-AVD)}_{\alpha}$ considered in \cite{HeHFiietal(2022)}. To the best of our knowledge, this is the first time  in the literature to consider Tikhonov regularized inertial primal-dual systems.

Following the approaches developed in \cite{AttouchZH2018, HeHFiietal(2022)}, we shall analyze the convergence properties of the trajectories generated by the dynamical system \eqref{z2} with $\alpha\ge 3$. Our main contributions are summarized as follows: For $\alpha\ge 3$,
\begin{itemize}
	\item[(a)] when   $\int_{t_{0}}^{+\infty}t\epsilon(t)dt<+\infty$ which reflects a fast vanishing $\epsilon(t)$, all of the primal-dual gap, the objective residual and the feasibility violation along the trajectories of \eqref{z2} own the $\mathcal{O}(\frac{1}{t^2})$ convergence rate.
	
	\item[(b)] under the condition $\int_{t_0}^{+\infty}\frac{\epsilon(t)}{t}dt<+\infty$, we derive a general minimization property of the trajectories generated by \eqref{z2}. Further, when  $\epsilon(t)=\frac{c}{t^r}$ with $0<r\le 2$ and $c>0$, the convergence rate depends upon $r$.

	\item[(c)]   when  $\lim_{t\to+\infty}t^2\epsilon(t)= +\infty$ which reflects a slowly vanishing $\epsilon(t)$, and  under the condition $\int_{t_0}^{+\infty}\frac{\epsilon(t)}{t}dt<+\infty$,  the primal trajectory $x(t)$ of \eqref{z2} converges strongly to the minimal norm solution of the problem \eqref{z1}.	
\end{itemize}

The rest of this paper is organized as follows: In section 2, we give some preliminary results. In section 3, we prove the existence and uniqueness of a strong global solution of the proposed dynamical system \eqref{z2}, and we also analyze the minimization properties of the trajectories generated by \eqref{z2}.  In section 4,  we prove the strong convergence of  the primal trajectory $x(t)$ generated by \eqref{z2} to the minimal norm solution of the problem \eqref{z1} when $\epsilon(t)$ decays slowly to zero. In section 5, we perform some numerical experiments to illustrate the efficiency of our approach and make some comparisons with the existing primal-dual dynamical approaches. 

\section{Preliminary results}

The Lagrangian function $\mathcal{L}: \mathcal{X}\times\mathcal{Y}\rightarrow\mathbb{R}$ associated with the problem \eqref{z1}  is defined by 
\begin{eqnarray*}
	\mathcal{L}(x,\lambda)=f(x)+\langle \lambda, Ax-b\rangle
\end{eqnarray*}
and the augmented Lagrangian function $\mathcal{L}_{\rho}: \mathcal{X}\times\mathcal{Y}\rightarrow\mathbb{R}$ is defined by 
\begin{eqnarray*}\label{zz1}
	\mathcal{L}_{\rho}(x,\lambda)=\mathcal{L}(x,\lambda)+\frac{\rho}{2}\|Ax-b\|^2
	=f(x)+\langle \lambda, Ax-b\rangle +\frac{\rho}{2}\|Ax-b\|^2,
\end{eqnarray*}
where $\rho\ge 0$ is the penalty parameter.
The associated dual problems for the problem \eqref{z1} are
\begin{eqnarray}\label{dudu5}
	\max_{\lambda\in\mathcal{Y}}d(\lambda)
\end{eqnarray}
and 
\begin{equation}\label{dp-fyp}
	\max_{\lambda\in\mathcal{Y}}d_{\rho}(\lambda),
\end{equation}
where 
$$d(\lambda)=\min_{x\in\mathcal{X}}\mathcal{L}(x,\lambda)$$and 
$$d_{\rho}(\lambda)=\min_{x\in\mathcal{X}}\mathcal{L}_{\rho}(x,\lambda).$$
If $x^*\in\mathcal{X}$ is a solution of the problem \eqref{z1} and $\lambda^*\in\mathcal{Y}$ is a solution of the problem \eqref{dudu5}, then $(x^*,\lambda^*)\in\mathcal{X}\times\mathcal{Y}$ is called a primal-dual optimal solution. In what follows, we always suppose that the primal-dual optimal solution set $\Omega$ associated with the problem \eqref{z1} is nonempty. As a consequence, the solution set $S$ of the problem \eqref{z1} is always nonempty. It is well-known that 
\begin{eqnarray}\label{zc3}
	(x^*, \lambda^*)\in\Omega&\Leftrightarrow& \left\{\begin{aligned}&\nabla f(x^*)+A^T \lambda^*&=0\\ &Ax^*-b&=0\end{aligned} \right.\nonumber\\
	&\Leftrightarrow& \left\{\begin{aligned}&\nabla f(x^*)+A^T \lambda^*+\rho A^T( Ax^*-b)&=0\\ &Ax^*-b&=0\end{aligned} \right.\nonumber\\
	&\Leftrightarrow&\left\{\begin{aligned}\nabla_{x}\mathcal{L}_{\rho}(x^*,\lambda^*)&=0\\
		\nabla_{\lambda}\mathcal{L}_{\rho}(x^*,\lambda^*)&=0\end{aligned} \right.
\end{eqnarray}
and $\mathcal{L}(x,\lambda)$ and $\mathcal{L}_{\rho}(x,\lambda)$ have the same saddle point set $\Omega$.

Given $\lambda\in\mathcal{Y}$, define
$$\mathcal{D}(\lambda)=\arg\min_{x\in\mathcal{X}}\mathcal{L}_{\rho}(x,\lambda).$$ 
Then, when $\rho>0$, $d_{\rho}(\lambda)$ is differentiable (see \cite[Chapter III: Remark 2.5]{FortinandGlowinski(1983)}) and 
$$\nabla d_{\rho}(\lambda)=Ax(\lambda)-b,$$
where $x(\lambda)$ is any element of $\mathcal{D}(\lambda)$. 
If $\lambda^*\in\mathcal{Y}$  is a solution of  the augmented Lagrangian dual problem \eqref{dp-fyp}, then $A(x(\lambda^*))=b$ for any $x(\lambda^*)\in \mathcal{D}(\lambda^*)$. This together with \eqref{zc3} implies
\begin{eqnarray}\label{ztt: onefang}
	\mathcal{D}(\lambda^*)=\arg\min_{x\in\mathcal{X}}\mathcal{L}_{\rho}(x,\lambda^*)\subseteq S,
\end{eqnarray}
where $\rho>0$.
The above  results in  the finite dimensional framework can be found in (\cite[Lemma 3.4]{ZLinandLiHandFang(2020)}, \cite[Lemma 2.1]{HongandLuo:(2017)}, \cite{HanDR(2022)}).

Next, we will list some lemmas which will be used in the convergence analysis.

\begin{lemma}\label{lemma2.1.1}\cite{XuWen2021}
	Let  $z\in \mathcal{X}$  and let $x: [t_0, +\infty)\rightarrow \mathcal{X}$ be a $\mathcal{C}^1$ function. If there exists a constant $\widetilde{C}>0$ such that
	$$\frac{1}{2}\left\|x(t)-z+\frac{t}{\alpha-1} \dot{x}(t)\right\|^2\leq \widetilde{C}, \quad \forall t\geq t_0,$$
	then $(x(t))_{t\geq t_0}$ is bounded.
\end{lemma}

\begin{lemma}\label{lemma2.1}\cite[Lemma 6]{HeHFiietal(2022)}
	Suppose that  $g:[t_{0},+\infty)\rightarrow\mathbb{R}^n$  is a continuous function and $a : [t_{0},+\infty)\rightarrow[0,+\infty)$ is a continuous function, where $t_{0}>0$.  If there exists  a constant $C\geq0$ such that
	$$\left\|g(t)+\int_{t_{0}}^{t}a(s)g(s)ds\right\|\leq C, \quad \forall t\geq t_{0},$$
	then
	$$\sup_{t\geq t_{0}}\|g(t)\|<+\infty.$$
\end{lemma}

\begin{lemma}\label{lemma2.2.3}\cite[Lemma A.3]{AttouchZH2018}
	Suppose that $\delta>0$, $\phi\in L^{1}([\delta, +\infty))$ is a nonnegative and continuous function, and $\psi :[\delta, +\infty)\rightarrow(0, +\infty)$ is a nondecreasing function such that $\lim_{t\rightarrow+\infty}\psi(t)=+\infty$. Then,
	$$\lim_{t\rightarrow+\infty}\frac{1}{\psi(t)}\int_{\delta}^{t}\psi(s)\phi(s)ds=0.$$
\end{lemma}

\section{Convergence properties of the solution trajectories}

For the beginning we will prove existence and uniqueness of a strong global solution of the dynamical system \eqref{z2}. 

\begin{theorem}\label{ztt4.1.1}
	Suppose that $\nabla f$ is $L$-Lipschitz continuous over $\mathcal{X}$ with $L>0$ and $\epsilon(t)\in L_{loc}^{1}([t_0, +\infty))$, where $ L_{loc}^{1}([t_0, +\infty))$ denotes the family of local integrable functions on $[t_0,+\infty)$. Then, for any given initial condition $(x(t_0), \lambda(t_0), \dot{x}(t_0))=(x_0, \lambda_0, u_0)\in \mathcal{X}\times\mathcal{Y}\times\mathcal{X}$, the dynamical system \eqref{z2} has a unique strong global solution.
\end{theorem}
\begin{proof}
Taking $u(t)=\dot{x}(t)$ and $Z(t)=(x(t), \lambda(t), u(t))$, we can rewrite \eqref{z2} as 
\begin{eqnarray}\label{dudu18}
	\begin{cases}
		\dot{Z}(t)+G(t, Z(t))=0,\\
		Z(t_0)=(x_0, \lambda_0, u_0),
	\end{cases}
\end{eqnarray}
where
\begin{eqnarray}\label{dudu15}
	G(t, Z(t))=\left(\begin{aligned}&-u\\
		&-t(A(x+\dfrac{t}{\alpha-1}u)-b)\\
		&\frac{\alpha}{t}u+\nabla f(x)+A^T\lambda+\rho A^T(Ax-b)+\epsilon(t)x\end{aligned} \right).
\end{eqnarray}
Since $\nabla f$ is $L$-Lipschitz continuous, it follows from \eqref{dudu15} that for any $Z, \bar{Z}\in \mathcal{X}\times\mathcal{Y}\times\mathcal{X}$, 
\begin{eqnarray*}
	\|G(t, Z)-G(t, \bar{Z})\|
	&\leq&(1+\frac{\alpha}{t}+\dfrac{t^2}{\alpha-1} \|A\|)\|u-\bar{u}\|
	+\|A^T\|\|\lambda-\bar{\lambda}\|\\
	&&+t\|A\|\|x-\bar{x}\|
	+\|\nabla f(x)-\nabla f(\bar{x})\|\\
	&&+\rho\|A^TA\|\|x-\bar{x}\|
	+\epsilon(t)\|x-\bar{x}\|\\
	&\leq&(C_1+\frac{\alpha}{t}+\dfrac{t^2}{\alpha-1}\|A\|+t\|A\|+\epsilon(t))\|Z-\bar{Z}\|,
\end{eqnarray*}
where $C_1=\max\{1, \|A^T\|, L+\rho\|A^TA\|\}$.
Taking $K(t)=C_1+\frac{\alpha}{t}+\dfrac{t^2}{\alpha-1}\|A\|+t\|A\|+\epsilon(t)$, we have
\begin{eqnarray*}
	\|G(t, Z)-G(t, \bar{Z})\|\leq K(t)\|Z-\bar{Z}\|,  \qquad \forall Z, \bar{Z}\in \mathcal{X}\times\mathcal{Y}\times\mathcal{X},
\end{eqnarray*}
where $K(t)\in L_{loc}^{1}([t_0, +\infty))$.
Because $\nabla f$ is $L$-Lipschitz continuous, we have
\begin{eqnarray*}
	\|\nabla f(x)-\nabla f(0)\|\leq L\|x\|\leq L\|Z\|,
\end{eqnarray*}
which together with \eqref{dudu15} implies
\begin{eqnarray*}
\|G(t, Z)\|
&\leq&(1+\frac{\alpha}{t}+\dfrac{t^2}{\alpha-1}\|A\|)\|Z\|
	+\|A^T\|\|Z\|+(t\|A\|+\rho\|A^TA\|+\epsilon(t))\|Z\|\\
	&&+t\|b\|+L\|Z\|+\|\nabla f(0)\|+\rho\|A^Tb\|\\
	&\leq&(1+\frac{\alpha}{t}+\dfrac{t^2}{\alpha-1}\|A\|)(1+\|Z\|)+\|A^T\|(1+\|Z\|)\\
	&&+(t\|A\|+\rho\|A^TA\|+\epsilon(t))(1+\|Z\|)\\
	&&+C_2(1+\|Z\|)+t\|b\|(1+\|Z\|)\\
	&\leq&(C_3+\frac{\alpha}{t}+\dfrac{t^2}{\alpha-1}\|A\|+t\|b\|)(1+\|Z\|)\\
	&&+(C_3+t\|A\|+\epsilon(t))(1+\|Z\|)+2C_3(1+\|Z\|),
\end{eqnarray*}
where $C_2=\max\{L, \|\nabla f(0)\|+\rho\|A^T b\|\}$ and $C_3=\max\{1, \rho\|A^TA\|, C_2, \|A^T\|\}$. Take
$$S(t)=4C_3+\frac{\alpha}{t}+\dfrac{t^2}{\alpha-1}\|A\|+t\|b\|+t\|A\|+\epsilon(t).$$ 
Then, we have
\begin{eqnarray*}
	\|G(t, Z)\|&\leq&S(t)(1+\|Z\|), \quad \forall Z\in\mathcal{X}\times\mathcal{Y}\times\mathcal{X},
\end{eqnarray*}
where $S(t)\in L_{loc}^{1}([t_0, +\infty))$. By \cite[Proposition 6.2.1]{HarauxAunique(1991)} and \cite[Theorem 5]{AttouchADMM(2022)}, the dynamical system \eqref{dudu18} has a unique global solution $Z(\cdot)\in W_{loc}^{1,1}([t_0, +\infty); \mathcal{X}\times\mathcal{Y}\times\mathcal{X})$ satisfying the initial condition $Z(t_0)=(x_0, \lambda_0, u_0)$, which together with \cite[Corollary A.2]{BrezisHilber1973} implies $Z(\cdot)$ is a strong global solution to \eqref{dudu18}. Thus, the dynamical system \eqref{z2} has a unique strong global solution $(x(\cdot), \lambda(\cdot))$.
\end{proof}

\subsection{Minimization properties of the trajectories}

In this subsection, depending on the speed of convergence of $\epsilon(t)$ to zero, we analyze the convergence properties of the trajectories generated by the dynamical system \eqref{z2}. The analysis is based on the approaches developed in \cite{AttouchZH2018, HeHFiietal(2022)}.

\begin{theorem}\label{ztt2.2}
Let $\epsilon: [t_0,+\infty)\rightarrow [0,+\infty)$ be a $\mathcal{C}^1$ nonincreasing function satisfying $\int_{t_{0}}^{+\infty}t\epsilon(t)dt<+\infty$ and let $(x,\lambda): [t_{0},+\infty)\rightarrow \mathcal{X}\times\mathcal{Y}$ be a solution of \eqref{z2}. For any $(x^*,\lambda^*)\in\Omega$, the following conclusions hold:
\begin{itemize}
\item[(i)] If $\alpha\geq3$, then the trajectory $(x(t),\lambda(t))_{t\geq t_{0}}$ is bounded and it holds
$$L_{\rho}(x(t),\lambda^*)-L_{\rho}(x^*,\lambda^*)=\mathcal{O}\left(\frac{1}{t^2}\right), \quad |f(x(t))-f(x^*)|=\mathcal{O}\left(\frac{1}{t^2}\right),$$
$$\|Ax(t)-b\|=\mathcal{O}\left(\frac{1}{t^2}\right), \quad
\|\dot{x}(t)\|=\mathcal{O}\left(\frac{1}{t}\right), \quad \|\nabla f(x(t))-\nabla f(x^*)\|=\mathcal{O}\left(\frac{1}{t}\right),$$
$$\rho\int_{t_{0}}^{+\infty}t\|Ax(t)-b\|^2dt<+\infty, \quad \int_{t_{0}}^{+\infty}t \epsilon(t)\|x(t)\|^2dt<+\infty.$$
\item[(ii)] If $\alpha>3$, it also holds
$$\int_{t_{0}}^{+\infty}t(L_{\rho}(x(t),\lambda^*)-L_{\rho}(x^*,\lambda^*))dt<+\infty,$$
$$\int_{t_{0}}^{+\infty}t\|\nabla f(x(t))-\nabla f(x^*)\|^2dt<+\infty.$$
\end{itemize}
\end{theorem}

\begin{proof}
Define  $\mathcal{E}:[t_0,+\infty)\to [0,+\infty)$ by
\begin{eqnarray}\label{z11}
	\mathcal{E}(t)&=&t^2(\mathcal{L}_{\rho}(x(t),\lambda^*)-\mathcal{L}_{\rho}(x^*,\lambda^*)+\frac{\epsilon(t)}{2}\|x(t)\|^2)\nonumber\\
	&&+\frac{1}{2}\|(\alpha-1)(x(t)-x^*)+t\dot{x}(t)\|^2+\frac{\alpha-1}{2}\|\lambda(t)-\lambda^*\|^2.
\end{eqnarray}
It follows that
\begin{eqnarray*}
	\dot{\mathcal{E}}(t)&=&2t(\mathcal{L}_{\rho}(x(t),\lambda^*)-\mathcal{L}_{\rho}(x^*,\lambda^*))
	+t^2\langle \nabla_{x}\mathcal{L}_{\rho}(x(t),\lambda^*),\dot{x}(t)\rangle
	+t^2\epsilon(t)\langle x(t), \dot{x}(t)\rangle\nonumber\\
	&&+(t\epsilon(t)+\frac{t^2}{2}\dot{\epsilon}(t))\|x(t)\|^2
	+\langle (\alpha-1)(x(t)-x^*)+t\dot{x}(t), \alpha\dot{x}(t)+t\ddot{x}(t)\rangle\\
	&&+(\alpha-1)\langle \lambda(t)-\lambda^*,\dot{\lambda}(t)\rangle.
\end{eqnarray*}
By using \eqref{z2}, we have
\begin{eqnarray*}
	&&\langle (\alpha-1)(x(t)-x^*)+t\dot{x}(t), \alpha\dot{x}(t)+t\ddot{x}(t)\rangle\\
	&&\quad=\langle(\alpha-1)(x(t)-x^*)+t\dot{x}(t), 
	-t\nabla_{x}\mathcal{L}_{\rho}(x(t),\lambda(t))
	-t\epsilon(t)x(t)\rangle\\
	&&\quad=-(\alpha-1)t\langle x(t)-x^*, \nabla_{x}\mathcal{L}_{\rho}(x(t),\lambda^*)\rangle
	-t^2\langle\dot{x}(t), \nabla_{x}\mathcal{L}_{\rho}(x(t),\lambda^*)\rangle\\
	&&\qquad-t\langle(\alpha-1)(x(t)-x^*)+t\dot{x}(t), A^T(\lambda(t)-\lambda^*)\rangle\\
	&&\qquad-(\alpha-1)t\epsilon(t)\langle x(t)-x^*, x(t)\rangle-t^2\epsilon(t)\langle x(t), \dot{x}(t)\rangle
\end{eqnarray*}
and
\begin{eqnarray*}
	\langle \lambda(t)-\lambda^*, \dot{\lambda}(t)\rangle=t\langle A^T(\lambda(t)-\lambda^*), x(t)-x^*+\dfrac{t}{\alpha-1}\dot{x}(t)\rangle.
\end{eqnarray*}
Then, we obtain
\begin{eqnarray}\label{zccz2}
	\dot{\mathcal{E}}(t)&=&2t(\mathcal{L}_{\rho}(x(t),\lambda^*)-\mathcal{L}_{\rho}(x^*,\lambda^*))
	+(t\epsilon(t)+\frac{t^2}{2}\dot{\epsilon}(t))\|x(t)\|^2\\
	&&-(\alpha-1)t\langle x(t)-x^*,\nabla_{x}\mathcal{L}_{\rho}(x(t),\lambda^*)\rangle
	-(\alpha-1)t\epsilon(t)\langle x(t)-x^*, x(t)\rangle.\nonumber
\end{eqnarray}
Since $f(x)+\frac{\epsilon(t)}{2}\|x\|^2$ is a $\epsilon(t)$\mbox{-}strongly convex function,  we have
\begin{eqnarray*}
	f(x^*)+\frac{\epsilon(t)}{2}\|x^*\|^2-f(x(t))-\frac{\epsilon(t)}{2}\|x(t)\|^2
	&\geq&\langle \nabla f(x(t))+\epsilon(t)x(t), x^*-x(t)\rangle\\
	&&+\frac{\epsilon(t)}{2}\|x(t)-x^*\|^2.
\end{eqnarray*}
This implies 
\begin{eqnarray}\label{zccz1}
	&&\langle\nabla_{x}\mathcal{L}_{\rho}(x(t),\lambda^*), x^*-x(t)\rangle\nonumber\\
	&&\quad=\langle \nabla f(x(t))+A^T\lambda^*+\rho A^T(Ax(t)-b),x^*-x(t)\rangle\nonumber\\
	&&\quad=\langle\nabla f(x(t)), x^*-x(t)\rangle-\langle \lambda^*,Ax(t)-b\rangle
	-\rho\|Ax(t)-b\|^2\nonumber\\
	&&\quad\leq\mathcal{L}_{\rho}(x^*,\lambda^*)-\mathcal{L}_{\rho}(x(t),\lambda^*)+\epsilon(t)\langle x(t),x(t)-x^*\rangle\nonumber\\
	&&\qquad+\frac{\epsilon(t)}{2}(\|x^*\|^2-\|x(t)\|^2-\|x(t)-x^*\|^2)-\frac{\rho}{2}\|Ax(t)-b\|^2,
\end{eqnarray}
which together with \eqref{zccz2} implies
\begin{eqnarray*}
	\dot{\mathcal{E}}(t)&\leq&(3-\alpha)t(\mathcal{L}_{\rho}(x(t),\lambda^*)-\mathcal{L}_{\rho}(x^*,\lambda^*))
	+\left[(1-\frac{\alpha-1}{2})t\epsilon(t)+\frac{t^2}{2}\dot{\epsilon}(t)\right]\|x(t)\|^2\nonumber\\
	&&+\frac{(\alpha-1)t\epsilon(t)}{2}(\|x^*\|^2-\|x(t)-x^*\|^2)-\frac{\rho(\alpha-1) t}{2}\|Ax(t)-b\|^2.
\end{eqnarray*}
Since $\alpha\geq3$ and  $\epsilon(t)$ is nonnegative and nonincreasing, we have
\begin{eqnarray*}
	&&\dot{\mathcal{E}}(t)-(3-\alpha)t(\mathcal{L}_{\rho}(x(t),\lambda^*)-\mathcal{L}_{\rho}(x^*,\lambda^*))
	+\frac{(\alpha-1)t\epsilon(t)}{2}\|x(t)-x^*\|^2\\
	&&\quad+\frac{\rho(\alpha-1) t}{2}\|Ax(t)-b\|^2
	\leq\frac{(\alpha-1)\|x^*\|^2}{2}t\epsilon(t).
\end{eqnarray*}

(i) Integrating the last inequality over  $[t_{0},t]$, we have
\begin{eqnarray}\label{ztt:aat}
	&&\mathcal{E}(t)+(\alpha-3)\int_{t_{0}}^{t}s(\mathcal{L}_{\rho}(x(s),\lambda^*)-\mathcal{L}_{\rho}(x^*,\lambda^*))ds+\frac{(\alpha-1)}{2}\int_{t_{0}}^{t}s\epsilon(s)\|x(s)-x^*\|^2ds\nonumber\\
	&&\quad+\frac{\rho(\alpha-1)}{2}\int_{t_{0}}^{t}s\|Ax(s)-b\|^2ds
	\leq\mathcal{E}(t_{0})+\frac{(\alpha-1)\|x^*\|^2}{2}\int_{t_{0}}^{t}s\epsilon(s)ds.
\end{eqnarray}
Since $\int_{t_{0}}^{+\infty}t\epsilon(t)dt<+\infty$, $\mathcal{E}(t)\geq0$ and  $\mathcal{L}_{\rho}(x(t),\lambda^*)-\mathcal{L}_{\rho}(x^*,\lambda^*)\geq0$, it follows from \eqref{ztt:aat} that $(\mathcal{E}(t))_{t\geq t_{0}}$ is bounded,
\begin{eqnarray*}\label{z14}
	\rho\int_{t_{0}}^{+\infty}t\|Ax(t)-b\|^2dt<+\infty
\end{eqnarray*}
and 
\begin{eqnarray*}
	\int_{t_{0}}^{+\infty}t\epsilon(t)\|x(t)-x^*\|^2dt<+\infty.
\end{eqnarray*}
Since $\int_{t_{0}}^{+\infty}t\epsilon(t)dt<+\infty$ and
$$t\epsilon(t)\|x(t)\|^2\leq 2t\epsilon(t)\|x(t)-x^*\|^2+2t\epsilon(t)\|x^*\|^2,$$
we obtain
\begin{eqnarray*}
	\int_{t_{0}}^{+\infty}t\epsilon(t)\|x(t)\|^2dt<+\infty.
\end{eqnarray*}
Further, by using \eqref{z11} and the boundedness of  $(\mathcal{E}(t))_{t\geq t_{0}}$, we have
\begin{eqnarray}\label{z17}
	\mathcal{L}_{\rho}(x(t),\lambda^*)-\mathcal{L}_{\rho}(x^*,\lambda^*)=\mathcal{O}\left(\frac{1}{t^2}\right),
\end{eqnarray}
the trajectory $(\lambda(t))_{t\geq t_0}$ is bounded, and
$$\left\{\frac{1}{2}\|x(t)-x^*+ \dfrac{t}{\alpha-1}\dot{x}(t)\|^2\right\}_{t\geq t_0}$$ is bounded.
By using Lemma \ref{lemma2.1.1}, we can obtain that $(x(t))_{t\geq t_0}$ is bounded. As a consequence,
$$\|\dot{x}(t)\|=\mathcal{O}\left(\frac{1}{t}\right).$$
Since $f$ is convex and $\nabla f$ is $L$-Lipschitz continuous, we have 
\begin{eqnarray*}
	f(x(t))-f(x^*)-\langle\nabla f(x^*), x(t)-x^*\rangle
	\geq\dfrac{1}{2L}\|\nabla f(x(t))-\nabla f(x^*)\|^2,
\end{eqnarray*}
which together with \eqref{zc3} implies
\begin{eqnarray}\label{ztt:sst}
	\mathcal{L}_{\rho}(x(t),\lambda^*)-\mathcal{L}_{\rho}(x^*,\lambda^*)
	&\geq& f(x(t))-f(x^*)+\langle\lambda^*, Ax(t)-b\rangle\nonumber\\
	&\geq&\langle\nabla f(x^*), x(t)-x^*\rangle+\langle A^T\lambda^*, x(t)-x^*\rangle\nonumber\\
	&&+\dfrac{1}{2L}\|\nabla f(x(t))-\nabla f(x^*)\|^2\nonumber\\
	&=&\dfrac{1}{2L}\|\nabla f(x(t))-\nabla f(x^*)\|^2.
\end{eqnarray}
This together with \eqref{z17} yields
\begin{eqnarray*}
	\|\nabla f(x(t))-\nabla f(x^*)\|=\mathcal{O}\left(\frac{1}{t}\right).	
\end{eqnarray*}
By using \eqref{z2}, we have
\begin{eqnarray*}
	\lambda(t)-\lambda(t_{0})
	&=&\int_{t_{0}}^{t}\dot{\lambda}(s)ds\nonumber\\
	&=&\int_{t_{0}}^{t}s\left[A\left(x(s)+\frac{s}{\alpha-1}\dot{x}(s)\right)-b\right]ds\nonumber\\
	&=&\int_{t_{0}}^{t}s(Ax(s)-b)ds
	+\int_{t_{0}}^{t}\frac{s^2}{\alpha-1}d(Ax(s)-b)\\
	&=&\frac{t^2(Ax(t)-b)}{\alpha-1}-\frac{t_{0}^2(Ax(t_{0})-b)}{\alpha-1}
	+\int_{t_{0}}^{t}\frac{(\alpha-3)}{\alpha-1}s(Ax(s)-b)ds,
\end{eqnarray*}
which together with the boundedness of $(\lambda(t))_{t\geq t_{0}}$ implies
\begin{eqnarray*}
	\left\|\frac{t^2(Ax(t)-b)}{\alpha-1}+\int_{t_{0}}^{t}\frac{(\alpha-3)}{s}\frac{s^2}{\alpha-1}(Ax(s)-b)ds\right\|
	\leq C_{4}, \quad \forall t\geq t_{0},
\end{eqnarray*}
where $C_{4}$ is a nonnegative constant.
Applying Lemma \ref{lemma2.1} with $g(t)=\frac{t^2(Ax(t)-b)}{\alpha-1}$ and $a(t)=\frac{(\alpha-3)}{t}$ to the last  inequality, we obtain
$$\sup_{t\geq t_{0}}\left\|\frac{t^2(Ax(t)-b)}{\alpha-1}\right\|<+\infty,$$
which yields
\begin{eqnarray}\label{z16}
	\|Ax(t)-b\|=\mathcal{O}\left(\frac{1}{t^2}\right).
\end{eqnarray}
Since
\begin{eqnarray*}
	\mathcal{L}_{\rho}(x(t),\lambda^*)-\mathcal{L}_{\rho}(x^*,\lambda^*)=f(x(t))-f(x^*)+\langle \lambda^*,Ax(t)-b\rangle+\frac{\rho}{2}\|Ax(t)-b\|^2,
\end{eqnarray*}
it follows from \eqref{z17} and \eqref{z16}  that
$$|f(x(t))-f(x^*)|=\mathcal{O}\left(\frac{1}{t^2}\right).$$

(ii) If $\alpha>3$, by using \eqref{ztt:aat}, we have 
\begin{eqnarray*}
	\int_{t_{0}}^{+\infty}t(\mathcal{L}_{\rho}(x(t),\lambda^*)-\mathcal{L}_{\rho}(x^*,\lambda^*))dt<+\infty.
\end{eqnarray*}
This together with \eqref{ztt:sst} implies
\begin{eqnarray*}
	\int_{t_{0}}^{+\infty}t\|\nabla f(x(t))-\nabla f(x^*)\|^2dt<+\infty.
\end{eqnarray*}	
\end{proof}

\begin{remark}\label{qq:sst}
	Under the condition  $\int_{t_{0}}^{+\infty}t\epsilon(t)dt<+\infty$, Attouch, Chbani, and Riahi  proved in \cite[Theorem 3.1]{AttouchZH2018} that $\text{(AVD)}_{\alpha, \epsilon}$ enjoys the fast convergence rates: (i) For $\alpha\ge 3$, $|f(x(t))-f(x^*)|=\mathcal{O}(\frac{1}{t^2})$. (ii) For $\alpha >3$,  $|f(x(t))-f(x^*)|={o}(\frac{1}{t^2})$ and the trajectory $x(t)$ converges weakly to a solution of the problem \eqref{up-f}. As a comparison, in Theorem \ref{ztt2.2} we can only obtain $|f(x(t))-f(x^*)|=\mathcal{O}(\frac{1}{t^2})$ for $\alpha\ge 3$. For $\alpha>3$, we cannot prove  $|f(x(t))-f(x^*)|=o(\frac{1}{t^2})$ and the weak convergence of $x(t)$ by using the approach developed in \cite[Theorem 3.1]{AttouchZH2018}, since the second-order term appears only in the primal variable of our dynamical system \eqref{z2}. When $\alpha>3$, some natural problems arise: 1) Is it true that the primal-dual gap, the objective function value and the feasibility measure along the trajectories of second-order plus first-order primal-dual dynamical systems own the $o(\frac{1}{t^2})$ convergence rate? 2) Whether the primal-dual trajectory of the second-order plus first-order primal-dual dynamical system converges weakly to a primal-dual optimal solution of the problem \eqref{z1}. The latter has been shown to be true for second-order plus second-order primal-dual dynamical system. See Bot and Nguyen \cite{BNguyen2022}.
\end{remark}

\begin{remark}
	Under the condition $\int_{t_0}^{+\infty} t\|\varepsilon(t)\|dt<+\infty$, it was shown in \cite[Theorem 2]{HeHFiietal(2022)} that $\text{(He-AVD)}_{\alpha}$ enjoys the fast convergence rates:  $|f(x(t))-f(x^*)|=\mathcal{O}(\frac{1}{t^2})$ and $\|Ax(t)-b\|=\mathcal{O}(\frac{1}{t^2})$ which are the same as the ones of Theorem \ref{ztt2.2}. However,  the convergence result of Theorem \ref{ztt2.2} is not a consequence of \cite[Theorem 2]{HeHFiietal(2022)} since we don't know a priori if the trajectory $(x(t))_{t\geq t_0}$ is bounded.
\end{remark}

\begin{remark}\label{corollary: eqfang}
	By taking $\epsilon(t)=\frac{c}{t^r}$ with $r>2$ and $c>0$, we have $\int_{t_0}^{+\infty}t\epsilon(t)<+\infty$. Thus, we can obtain all conclusions of Theorem \ref{ztt2.2}.
\end{remark}

In Theorem \ref{ztt2.2}, we establish the fast convergence properties of \eqref{z2} when $\int_{t_{0}}^{+\infty}t\epsilon(t)dt<+\infty$ which reflects the case that the Tikhonov regularization parameter $\epsilon(t)$ decreases rapidly to zero. Next, we analyze the minimization properties of \eqref{z2} under the condition $\int_{t_0}^{+\infty}\frac{\epsilon(t)}{t}dt<+\infty$.

\begin{lemma}\label{lemma2.2}
Let $(x,\lambda): [t_{0},+\infty)\rightarrow \mathcal{X}\times\mathcal{Y}$ be a solution of \eqref{z2} and $(x^*,\lambda^*)\in\Omega$. Denote	
\begin{eqnarray}\label{z3}
	\widetilde{\mathcal{E}} (t)&=&(\mathcal{L}_{\rho}(x(t),\lambda^*)-\mathcal{L}_{\rho}(x^*,\lambda^*))+\frac{\epsilon(t)}{2}\|x(t)\|^2
	\nonumber\\
	&&+\frac{1}{2}\|\frac{\alpha-1}{t}(x(t)-x^*)+\dot{x}(t)\|^2+\frac{\alpha-1}{2t^2}\|\lambda(t)-\lambda^*\|^2.
\end{eqnarray}
For every $t\geq t_{0}$, it holds
\begin{eqnarray*}
	\frac{2}{t}\widetilde{\mathcal{E}} (t)+\dot{\widetilde{\mathcal{E}}}(t)
	&\leq&\frac{(3-\alpha)}{t}(\mathcal{L}_{\rho}(x(t),\lambda^*)-\mathcal{L}_{\rho}(x^*,\lambda^*))\\
	&&+\frac{1}{2t}\left[(3-\alpha)\epsilon(t)+t\dot{\epsilon}(t)\right]\|x(t)\|^2+\frac{(\alpha-1)\|x^*\|^2}{2}\dfrac{\epsilon(t)}{t}\\
	&&-\frac{(\alpha-1)\epsilon(t)}{2t}\|x(t)-x^*\|^2\nonumber
	-\frac{\rho(\alpha-1)}{2t}\|Ax(t)-b\|^2, \quad \forall t\geq t_{0}.
\end{eqnarray*}
\end{lemma}
\begin{proof}
By the definitions of $u(t)$ and $\widetilde{\mathcal{E}}(t)$, we have
\begin{eqnarray*}
	\dot{\widetilde{\mathcal{E}}}(t)
	&=&\langle\nabla_{x}\mathcal{L}_{\rho}(x(t),\lambda^*),\dot{x}(t)\rangle+\epsilon(t)\langle x(t),\dot{x}(t)\rangle+\frac{\dot{\epsilon}(t)}{2}\|x(t)\|^2\\
	&&+\langle\frac{\alpha-1}{t}(x(t)-x^*)+\dot{x}(t),-\frac{(\alpha-1)}{ t^2}(x(t)-x^*)
	+\frac{\alpha-1}{t}\dot{x}(t)+\ddot{x}(t)\rangle\\
	&&+\frac{\alpha-1}{t^2}\langle \lambda(t)-\lambda^*,\dot{\lambda}(t)\rangle
	 -\frac{(\alpha-1)}{ t^3}\|\lambda(t)-\lambda^*\|^2.
\end{eqnarray*}
By using \eqref{z2}, we have
\begin{eqnarray*}
	&&\langle \frac{\alpha-1}{t}(x(t)-x^*)+\dot{x}(t),-\frac{(\alpha-1)}{ t^2}(x(t)-x^*)
	+\frac{\alpha-1}{t}\dot{x}(t)+\ddot{x}(t)\rangle\\
	&&\quad=\langle\frac{\alpha-1}{t}(x(t)-x^*)+\dot{x}(t),-\frac{(\alpha-1)}{ t^2}(x(t)-x^*)
	-\frac{1}{t}\dot{x}(t)
	-\nabla_{x}\mathcal{L}_{\rho}(x(t),\lambda(t))\rangle\\
	&&\qquad-\langle\frac{\alpha-1}{t}(x(t)-x^*)+\dot{x}(t), \epsilon(t)x(t)\rangle\\
	&&\quad=-\frac{(\alpha-1)^2}{t^3}\|x(t)-x^*\|^2- \frac{2(\alpha-1)}{t^2}\langle x(t)-x^*, \dot{x}(t)\rangle-\frac{1}{t}\|\dot{x}(t)\|^2\\
	&&\qquad-\frac{(\alpha-1)}{t}\langle x-x^*, \nabla_{x}\mathcal{L}_{\rho}(x(t),\lambda^*)\rangle
	-\langle \dot{x}(t), \nabla_{x}\mathcal{L}_{\rho}(x(t),\lambda^*)\rangle\\
	&&\qquad-\langle \frac{\alpha-1}{t}(x(t)-x^*)+\dot{x}(t), A^T(\lambda(t)-\lambda^*)\rangle\\
	&&\qquad-\frac{(\alpha-1)\epsilon(t)}{t}\langle x(t)-x^*, x(t)\rangle
	-\epsilon(t)\langle x(t), \dot{x}(t)\rangle
\end{eqnarray*}
and
\begin{eqnarray*}
	\langle \lambda(t)-\lambda^*, \dot{\lambda}(t)\rangle=t\langle A^T(\lambda(t)-\lambda^*), x(t)-x^*+\frac{t}{\alpha-1}\dot{x}(t)\rangle.
\end{eqnarray*}
Thus, we have
\begin{eqnarray*}\label{z4}
	\dot{\widetilde{\mathcal{E}}}(t)&=&\frac{\dot{\epsilon}(t)}{2}\|x(t)\|^2-\frac{(\alpha-1)^2}{t^3}\|x(t)-x^*\|^2-\frac{2(\alpha-1)}{t^2}\langle x(t)-x^*,\dot{x}(t)\rangle\nonumber\\
	&&-\frac{1}{t}\|\dot{x}(t)\|^2
	-\frac{(\alpha-1)}{t}\langle x-x^*, \nabla_{x}\mathcal{L}_{\rho}(x(t),\lambda^*)\rangle\nonumber\\
	&&-\frac{(\alpha-1)\epsilon(t)}{t}\langle x(t)-x^*,x(t)\rangle
	-\frac{(\alpha-1)}{t^3}\|\lambda(t)-\lambda^*\|^2,
\end{eqnarray*}
which together with \eqref{zccz1} implies
\begin{eqnarray*}\label{z5}
	\dot{\widetilde{\mathcal{E}}}(t)&\leq&\left(\frac{\dot{\epsilon}(t)}{2}-\frac{(\alpha-1)\epsilon(t)}{2 t}\right)\|x(t)\|^2
	-\frac{(\alpha-1)}{2 t^3}(2(\alpha-1)
	+t^2\epsilon(t))\|x(t)-x^*\|^2\nonumber\\
	&&+\frac{(\alpha-1)\epsilon(t)}{2 t}\|x^*\|^2-\frac{1}{t}\|\dot{x}(t)\|^2
	-\frac{(\alpha-1)}{ t}(\mathcal{L}_{\rho}(x(t),\lambda^*)-\mathcal{L}_{\rho}(x^*,\lambda^*))\nonumber\\
	&&-\frac{2(\alpha-1)}{t^2}\langle x(t)-x^*,\dot{x}(t)\rangle
	-\frac{(\alpha-1)}{ t^3}\|\lambda(t)-\lambda^*\|^2\\
	&&-\frac{\rho(\alpha-1)}{2t}\|Ax(t)-b\|^2.\nonumber
\end{eqnarray*}
According to \eqref{z3}, we obtain 
\begin{eqnarray*}
	\widetilde{\mathcal{E}}(t)&=&(\mathcal{L}_{\rho}(x(t),\lambda^*)-\mathcal{L}_{\rho}(x^*,\lambda^*))+\frac{\epsilon(t)}{2}\|x(t)\|^2+\frac{(\alpha-1)^2}{2t^2}\|x(t)-x^*\|^2\\
	&&+\frac{1}{2}\|\dot{x}(t)\|^2
	+\frac{\alpha-1}{t}\langle x(t)-x^*,\dot{x}(t)\rangle+\frac{\alpha-1}{2 t^2}\|\lambda(t)-\lambda^*\|^2.
\end{eqnarray*}
Thus, we have 
\begin{eqnarray*}
	\frac{2}{t}\widetilde{\mathcal{E}}(t)+\dot{\widetilde{\mathcal{E}}}(t)
	&\leq&\frac{(3-\alpha)}{t}(\mathcal{L}_{\rho}(x(t),\lambda^*)
	-\mathcal{L}_{\rho}(x^*,\lambda^*))\\
	&&+\frac{1}{2t}\left[(3-\alpha)\epsilon(t)+t\dot{\epsilon}(t)\right]\|x(t)\|^2+\frac{(\alpha-1)\|x^*\|^2}{2}\frac{\epsilon(t)}{t}\\
	&&-\frac{(\alpha-1)\epsilon(t)}{2t}\|x(t)-x^*\|^2\nonumber
	-\frac{\rho(\alpha-1)}{2t}\|Ax(t)-b\|^2, \quad \forall t\geq t_{0}.
\end{eqnarray*} 	
\end{proof}

Based on Lemma \ref{lemma2.2}, we discuss the minimization properties of \eqref{z2}.

\begin{theorem}\label{theorem:ztt}
	Suppose that $\alpha\geq3$ and $\epsilon: [t_0,+\infty)\rightarrow\mathbb{R}^+$ is a $\mathcal{C}^1$ nonincreasing function satisfying 
	$\int_{t_{0}}^{+\infty}\frac{\epsilon(t)} {t}dt<+\infty$.
	Let $(x,\lambda): [t_{0},+\infty)\rightarrow \mathcal{X}\times\mathcal{Y}$ be a solution of \eqref{z2}. For any $(x^*,\lambda^*)\in\Omega$, the following conclusions hold: 
	$$\lim_{t\rightarrow+\infty}\mathcal{L}_{\rho}(x(t),\lambda^*)-\mathcal{L}_{\rho}(x^*,\lambda^*)=0,$$
	$$\lim_{t\rightarrow+\infty}\|\frac{\alpha-1}{ t}(x(t)-x^*)+\dot{x}(t)\|=0.$$
\end{theorem}

\begin{proof}
Since $\alpha\geq3$, $\dot{\epsilon}(t)\leq0$ and  $\mathcal{L}_{\rho}(x(t),\lambda^*)-\mathcal{L}_{\rho}(x^*,\lambda^*)\geq0$, it follows from Lemma \ref{lemma2.2} that 
\begin{eqnarray*}
	\frac{2}{t}\widetilde{\mathcal{E}}(t)+\dot{\widetilde{\mathcal{E}}}(t)\leq\frac{(\alpha-1)\|x^*\|^2}{2}\frac{\epsilon(t)}{t}, \quad \forall t\geq t_{0}.
\end{eqnarray*}
Multiplying both sides of the above inequality by $t^2$, we have
\begin{eqnarray*}\label{z9}
	\frac{d}{dt}\left(t^{2}\widetilde{\mathcal{E}}(t)\right)\leq\frac{(\alpha-1)\|x^*\|^2}{2}t\epsilon(t), \quad \forall t\geq t_{0}.
\end{eqnarray*}
Integrating it from $t_{0}$ to $t$, we obtain
\begin{eqnarray*}
	t^{2}\widetilde{\mathcal{E}}(t)\leq t_{0}^{2}\widetilde{\mathcal{E}}(t_0)
	+\frac{(\alpha-1)\|x^*\|^2}{2}\int_{t_0}^{t}s\epsilon(s)ds.
\end{eqnarray*}
Dividing both sides of the last inequality  by $t^{2}$, we have
\begin{eqnarray}\label{dudu3}
	\widetilde{\mathcal{E}}(t)\leq \frac{t_{0}^{2}\widetilde{\mathcal{E}}(t_0)}{t^{2}}
	+\frac{(\alpha-1)\|x^*\|^2}{2t^2}\int_{t_0}^{t}s^2\frac{\epsilon(s)}{s}ds.
\end{eqnarray}
Because of $\int_{t_0}^{+\infty}\frac{\epsilon(t)}{t}dt<+\infty$,
applying Lemma \ref{lemma2.2.3} with $\psi(t)=t^{2}$ and $\phi(t)=\frac{\epsilon(t)}{t}$ to \eqref{dudu3}, we have
\begin{eqnarray*}
	\lim_{t\rightarrow+\infty}\frac{1}{ t^{2}}\int_{t_0}^{t}s^{2}\frac{\epsilon(s)}{s}ds=0,
\end{eqnarray*}
which together with $\widetilde{\mathcal{E}}(t)\geq0$ implies 
\begin{eqnarray*}
	\lim_{t\rightarrow+\infty}\widetilde{\mathcal{E}}(t)=0.
\end{eqnarray*}
This together with \eqref{z3} implies 
\begin{eqnarray*}
	\lim_{t\rightarrow+\infty}\mathcal{L}_{\rho}(x(t),\lambda^*)-\mathcal{L}_{\rho}(x^*,\lambda^*)=0,
\end{eqnarray*}
\begin{eqnarray*}
\lim_{t\rightarrow+\infty}\|\frac{\alpha-1}{ t}(x(t)-x^*)+\dot{x}(t)\|=0.
\end{eqnarray*}
The proof is complete.	
\end{proof}

Further, when  $\epsilon(t)=\frac{c}{t^r}$ with $0<r\leq2$ and $c>0$, we will show that the convergence rates of the trajectories of \eqref{z2} depend upon $r$.

\begin{theorem}\label{dudu2.2}
Suppose that $\alpha\geq3$ and $\epsilon(t)=\frac{c}{t^r}$ with $0<r\leq2$ and $c>0$. Let $(x,\lambda): [t_{0},+\infty)\rightarrow \mathcal{X}\times\mathcal{Y}$ be a solution of \eqref{z2}. For any $(x^*,\lambda^*)\in\Omega$, the following conclusions hold:
\begin{itemize}
	\item[(i)] If $0<r<2$, then the trajectory $(x(t))_{t\geq t_{0}}$ is bounded, and  it holds
	$$\mathcal{L}_{\rho}(x(t),\lambda^*)-\mathcal{L}_{\rho}(x^*,\lambda^*)=\mathcal{O}\left(\frac{1}{t^r}\right),$$
	$$\|\frac{\alpha-1}{ t}(x(t)-x^*)+\dot{x}(t)\|=\mathcal{O}\left(\frac{1}{t^{\frac{r}{2}}}\right), \quad \left\|\dot{x}(t)\right\|=\mathcal{O}\left(\frac{1}{t^{\frac{r}{2}}}\right).$$
	\item[(ii)] If $r=2$, it holds
	$$\mathcal{L}_{\rho}(x(t),\lambda^*)-\mathcal{L}_{\rho}(x^*,\lambda^*)=\mathcal{O}\left(\frac{\ln t}{t^2}\right),$$
	$$\|\frac{\alpha-1}{t}(x(t)-x^*)+\dot{x}(t)\|=\mathcal{O}\left(\frac{\sqrt{\ln t}}{t}\right).$$
\end{itemize}
\end{theorem}

\begin{proof}
(i) By using \eqref{z3} and $\epsilon(t)=\frac{c}{t^r}$ with $0<r<2$, we have
\begin{eqnarray}\label{qq:ztt}
	\widetilde{\mathcal{E}} (t)&=&(\mathcal{L}_{\rho}(x(t),\lambda^*)-\mathcal{L}_{\rho}(x^*,\lambda^*))+\frac{c}{2t^r}\|x(t)\|^2\nonumber\\
	&&+\frac{1}{2}\|\frac{\alpha-1}{t}(x(t)-x^*)+\dot{x}(t)\|^2+\frac{\alpha-1}{2t^2}\|\lambda(t)-\lambda^*\|^2. 
\end{eqnarray}
Because of $0<r<2$ and $\alpha\geq3$, it is easy to verify that all conditions of Theorem \ref{theorem:ztt} are satisfied. Using the method similar to proving \eqref{dudu3}, we can obtain
\begin{eqnarray*}
	\widetilde{\mathcal{E}}(t)&\leq& \frac{t_{0}^{2}\widetilde{\mathcal{E}}(t_0)}{t^{2}}+\frac{c(\alpha-1)\|x^*\|^2}{2t^{2}}\int_{t_0}^{t}s^{1-r}ds\\
	&\leq&\frac{t_{0}^{2}\widetilde{\mathcal{E}}(t_0)}{t^{2}}+\frac{c(\alpha-1)\|x^*\|^2}{2(2-r)t^r}.
\end{eqnarray*} 
Multiplying both sides of the last inequality by $t^r$, we have
\begin{eqnarray*}
	t^r\widetilde{\mathcal{E}}(t)\leq\frac{t_{0}^{2}\widetilde{\mathcal{E}}(t_0)}{t^{2-r}}+\frac{c(\alpha-1)\|x^*\|^2}{2(2-r)}, \quad \forall t\geq t_{0}.
\end{eqnarray*}
This together with  \eqref{qq:ztt} implies that the trajectory $(x(t))_{t\geq t_{0}}$ is bounded,
$$\mathcal{L}_{\rho}(x(t),\lambda^*)-\mathcal{L}_{\rho}(x^*,\lambda^*)=\mathcal{O}\left(\frac{1}{t^r}\right),$$
$$\|\frac{\alpha-1}{t}(x(t)-x^*)+\dot{x}(t)\|=\mathcal{O}\left(\frac{1}{t^{\frac{r}{2}}}\right)$$ 
and
$$\|\dot{x}(t)\|=\mathcal{O}\left(\frac{1}{t^{\frac{r}{2}}}\right).$$

(ii) By using \eqref{z3} and $\epsilon(t)=\frac{c}{t^2}$, we can obtain
\begin{eqnarray}\label{qq:vtt}
	\quad\widetilde{\mathcal{E}} (t)&=&(\mathcal{L}_{\rho}(x(t),\lambda^*)-\mathcal{L}_{\rho}(x^*,\lambda^*))+\frac{c}{2t^2}\|x(t)\|^2\nonumber\\
	&&+\frac{1}{2}\|\frac{\alpha-1}{t}(x(t)-x^*)+\dot{x}(t)\|^2+\frac{\alpha-1}{2t^2}\|\lambda(t)-\lambda^*\|^2. 
\end{eqnarray}
Since $r=2$ and $\alpha\geq3$,  all conditions of Theorem \ref{theorem:ztt} are satisfied.
Using again the method similar to proving \eqref{dudu3}, we have
\begin{eqnarray*}
	\widetilde{\mathcal{E}}(t)&\leq& \frac{t_{0}^{2}\widetilde{\mathcal{E}}(t_0)}{t^{2}}
	+\frac{(\alpha-1)\|x^*\|^2}{2t^2}\int_{t_0}^{t}\frac{1}{s}ds\\
	&\leq&\frac{t_{0}^{2}\widetilde{\mathcal{E}}(t_0)}{t^{2}}
	+\frac{(\alpha-1)\|x^*\|^2 \ln t}{2t^2}-\frac{(\alpha-1)\|x^*\|^2 \ln t_0}{2t^2}\\
	&\leq&\dfrac{C_4}{2t^2}+\frac{(\alpha-1)\|x^*\|^2 \ln t}{2t^2},
\end{eqnarray*}
where $C_4\geq 2t_{0}^{2}\widetilde{\mathcal{E}}(t_0)-(\alpha-1)\|x^*\|^2 \ln t_0$ is a nonnegative constant. This together with  \eqref{qq:vtt} implies
$$\mathcal{L}_{\rho}(x(t),\lambda^*)-\mathcal{L}_{\rho}(x^*,\lambda^*)=\mathcal{O}\left(\frac{\ln t}{t^2}\right)$$
and
$$\|\frac{\alpha-1}{ t}(x(t)-x^*)+\dot{x}(t)\|=\mathcal{O}\left(\frac{\sqrt{\ln t}}{t}\right).$$	
\end{proof}

\section{Strong convergence of the trajectory to the minimal norm solution}

In this section, we will investigate the strong convergence of  the trajectory $(x(t))_{t\geq t_{0}}$ generated by \eqref{z2}  to the minimal norm  solution of \eqref{z1} when the Tikhonov regularization parameter $\epsilon(t)$ decreases slowly to zero. Throughout this section, we assume $\rho>0$.

Let $\bar{x}^*\in\mathcal{X}$ be the minimal norm solution of the problem  \eqref{z1}, that is $\bar{x}^*=Proj_{S}0$, where $S$ is the solution set of  the problem \eqref{z1} and   $Proj$ means the projection operator. Then, there exists an optimal solution $\bar{\lambda}^*\in\mathcal{Y}$ of the Lagrange dual problem \eqref{dudu5} such that $(\bar{x}^*,\bar{\lambda}^*)\in\Omega$. For any $\epsilon>0$, define $\mathcal{L}_{\epsilon} : \mathcal{X}\rightarrow\mathbb{R}$  by
\begin{eqnarray}\label{z18}
	\mathcal{L}_{\epsilon}(x):=\mathcal{L}_{\rho}(x,\bar{\lambda}^*)+\frac{\epsilon}{2}\|x\|^2,\quad \forall x\in  \mathcal{X}.
\end{eqnarray}
Clearly,  $\mathcal{L}_{\epsilon}$ is strongly convex and has a unique minimizer.
Set
$$x_{\epsilon}:=\arg\min_{x\in\mathcal{X}}\mathcal{L}_{\epsilon}(x).$$
By using the first-order optimality condition, we obtain
\begin{eqnarray}\label{z19}
	\nabla\mathcal{L}_{\epsilon}(x_{\epsilon})=\nabla_{x}\mathcal{L}_{\rho}(x_{\epsilon},\bar{\lambda}^*)+\epsilon x_{\epsilon}=0.
\end{eqnarray}
It follows from $(\bar{x}^*, \bar{\lambda}^*)\in\Omega$ and \eqref{zc3} that $\bar{x}^*\in \mathcal{D}(\bar{\lambda}^*)$. Since $\bar{x}^*=Proj_{S}0$, from \eqref{ztt: onefang} we have
\begin{eqnarray*}
	\bar{x}^*=Proj_{\mathcal{D}(\bar{\lambda}^*)}0.
\end{eqnarray*}
The classical properties of Tikhonov regularization give
\begin{eqnarray*}
	\|x_{\epsilon}\|\leq\|\bar{x}^*\|, \quad \forall \epsilon> 0
\end{eqnarray*}
and
\begin{eqnarray}\label{z26}
	\lim_{\epsilon\rightarrow0}\|x_{\epsilon}-\bar{x}^*\|=0.
\end{eqnarray}
For details of the results of Tikhonov regularization we refer to (\cite{Attouch(1996)},\cite{AttouchCominetti1996}).

\begin{lemma}\label{lemmaztt3.3}
Let  $\bar{x}^*=Proj_{S}0$ and let  $\epsilon: [t_0, +\infty)\rightarrow [0,+\infty)$ be a $\mathcal{C}^1$ nonincreasing function satisfying $\lim_{t\rightarrow+\infty}\epsilon(t)=0$.
Let $(x,\lambda): [t_{0},+\infty)\rightarrow \mathcal{X}\times\mathcal{Y}$ be a solution of \eqref{z2}.
Then, 
\begin{eqnarray*}
\frac{\epsilon(t)}{2}(\|x(t)-x_{\epsilon(t)}\|^2+\|x_{\epsilon(t)}\|^2-\|\bar{x}^*\|^2)\leq\mathcal{L}_{\epsilon(t)}(x(t))-\mathcal{L}_{\epsilon(t)}(\bar{x}^*).
\end{eqnarray*}
\end{lemma}

\begin{proof}
By using definition, we can obtain that $\mathcal{L}_{\epsilon(t)}$ is $\epsilon(t)$-strongly convex. Thus, we have
\begin{eqnarray*}
	\mathcal{L}_{\epsilon(t)}(x(t))-\mathcal{L}_{\epsilon(t)}(x_{\epsilon(t)})\geq\langle \nabla\mathcal{L}_{\epsilon(t)}(x_{\epsilon(t)}), x(t)-x_{\epsilon(t)}\rangle+\frac{\epsilon(t)}{2}\|x(t)-x_{\epsilon(t)}\|^2,
\end{eqnarray*}
which together with \eqref{z19} implies
\begin{eqnarray*}
	\mathcal{L}_{\epsilon(t)}(x(t))-\mathcal{L}_{\epsilon(t)}(x_{\epsilon(t)})\geq \frac{\epsilon(t)}{2}\|x(t)-x_{\epsilon(t)}\|^2.
\end{eqnarray*}
By using \eqref{z18}, we have
\begin{eqnarray*}
	\mathcal{L}_{\epsilon(t)}(\bar{x}^*)-\mathcal{L}_{\epsilon(t)}(x_{\epsilon(t)})&=&\mathcal{L}_{\rho}(\bar{x}^*,\bar{\lambda}^*)
	-\mathcal{L}_{\rho}(x_{\epsilon(t)},\bar{\lambda}^*)
	+\frac{\epsilon(t)}{2}\|\bar{x}^*\|^2-\frac{\epsilon(t)}{2}\|x_{\epsilon(t)}\|^2\\
	&\leq&\frac{\epsilon(t)}{2}\|\bar{x}^*\|^2-\frac{\epsilon(t)}{2}\|x_{\epsilon(t)}\|^2,
\end{eqnarray*}
where the last inequality uses the fact that $\mathcal{L}_{\rho}(\bar{x}^*,\bar{\lambda}^*)
-\mathcal{L}_{\rho}(x_{\epsilon},\bar{\lambda}^*)\leq0$.
As a consequence,
\begin{eqnarray*}
	\frac{\epsilon(t)}{2}(\|x(t)-x_{\epsilon(t)}\|^2+\|x_{\epsilon(t)}\|^2-\|\bar{x}^*\|^2)
	\leq\mathcal{L}_{\epsilon(t)}(x(t))-\mathcal{L}_{\epsilon(t)}(\bar{x}^*).
\end{eqnarray*}	
\end{proof}

\begin{theorem}\label{theoremztt3.1}
	Suppose that $\alpha \geq3$ and $\epsilon: [t_0, +\infty)\rightarrow\mathbb{R}^+$ is a $\mathcal{C}^1$ nonincreasing function such that
	\begin{itemize}
		\item[(i)] $\lim_{t\rightarrow+\infty}t^2\epsilon(t)=+\infty$;
		\item[(ii)] $\int_{t_0}^{+\infty}\frac{\epsilon(t)}{t}dt<+\infty$.
	\end{itemize}
	Let $(x,\lambda): [t_{0},+\infty)\rightarrow \mathcal{X}\times\mathcal{Y}$ be a solution of \eqref{z2}. Then,
	$$\liminf_{t\rightarrow+\infty}\|x(t)-\bar{x}^*\|=0,$$
	where  $\bar{x}^*=proj_{S}0$.
	Further, if there exists a constant $T>0$ such that the trajectory $\{x(t): t\geq T\}$ stays in either the open ball $B(0,\|\bar{x}^*\|)$ or its complement, then $$\lim_{t\rightarrow+\infty}\|x(t)-\bar{x}^*\|=0.$$
\end{theorem}

\begin{proof}
Depending upon the sign of the term $\|\bar{x}^*\|-\|x(t)\|$, we analyze separately the following three situations. 

{\bf Case I:} There exists a large enough $T$ such that $\{x(t): t\geq T\}$ stays in the complement of $B(0,\|\bar{x}^*\|)$. In this case, 
\begin{eqnarray}\label{dudu8}
	\|x(t)\|\geq\|\bar{x}^*\|, \quad \forall t\geq T.
\end{eqnarray}
Since $\bar{x}^*\in\mathcal{X}$ is the minimal norm solution of the problem \eqref{z1},  there exists an optimal solution $\bar{\lambda}^*\in\mathcal{Y}$ of the Lagrange dual problem \eqref{dudu5} such that $(\bar{x}^*,\bar{\lambda}^*)\in\mathcal{X}\times\mathcal{Y}$ is a primal-dual optimal solution, i.e., $(\bar{x}^*,\bar{\lambda}^*)\in\Omega$.

Denote
\begin{eqnarray}\label{z24}
	\widehat{\mathcal{E}}(t)&=&\mathcal{L}_{\epsilon(t)}(x(t))-\mathcal{L}_{\epsilon(t)}(\bar{x}^*)
	+\frac{1}{2}\|\frac{\alpha-1}{ t}(x(t)-\bar{x}^*)+\dot{x}(t)\|^2\nonumber\\
	&&+\frac{\alpha-1}{2t^2}\|\lambda(t)-\bar{\lambda}^*\|^2.
\end{eqnarray}
Using the method similar to the proof of Lemma \ref{lemma2.2}, we obtain
\begin{eqnarray*}
	\frac{2}{t}\widehat{\mathcal{E}} (t)+\dot{\widehat{\mathcal{E}}}(t)
	&\leq&\frac{(3-\alpha)}{t}(\mathcal{L}_{\rho}(x(t),\bar{\lambda}^*)
	-\mathcal{L}_{\rho}(\bar{x}^*,\bar{\lambda}^*))\\
	&&+\frac{1}{2t}\left[t\dot{\epsilon}(t)+(3-\alpha)\epsilon(t)\right](\|x(t)\|^2-\|\bar{x}^*\|^2)\\
	&&-\frac{(\alpha-1)\epsilon(t)}{2t}\|x(t)-\bar{x}^*\|^2-\frac{\rho(\alpha-1)}{2t}\|Ax(t)-b\|^2, \quad \forall t\geq t_{0}.
\end{eqnarray*}
Since $\alpha\geq3$ and  $\mathcal{L}_{\rho}(x(t),\bar{\lambda}^*)
-\mathcal{L}_{\rho}(\bar{x}^*,\bar{\lambda}^*)\geq0$, we have 
\begin{eqnarray*}
\frac{2}{t}\widehat{\mathcal{E}} (t)+\dot{\widehat{\mathcal{E}}}(t)\leq
\frac{1}{2t}\left[t\dot{\epsilon}(t)+(3-\alpha)\epsilon(t)\right](\|x(t)\|^2-\|\bar{x}^*\|^2), \quad \forall t\geq t_{0},
\end{eqnarray*}
which together with $\dot{\epsilon}(t)\leq0$ and \eqref{dudu8} implies 
\begin{eqnarray*}
	\frac{2}{t}\widehat{\mathcal{E}} (t)+\dot{\widehat{\mathcal{E}}}(t)&\leq&0, \quad \forall t\geq T.
\end{eqnarray*}
Multiplying both sides of the above inequality by $t^{2}$, we have
\begin{eqnarray*}
	\frac{d}{dt}\left(t^{2}\widehat{\mathcal{E}}(t)\right)\leq0, \quad \forall t\geq T.
\end{eqnarray*}
Integrating it  from $T$ to $t$, we have
\begin{eqnarray*}
	t^{2}\widehat{\mathcal{E}}(t)\leq T^2\widehat{\mathcal{E}}(T), \quad \forall t\geq T.
\end{eqnarray*}
As a consequence,
\begin{eqnarray*}
	\widehat{\mathcal{E}}(t)\leq \frac{T^2\widehat{\mathcal{E}}(T)}{t^{2}}, \quad \forall t\geq T.
\end{eqnarray*}
By using \eqref{z24}, we can obtain
\begin{eqnarray*}
	\mathcal{L}_{\epsilon(t)}(x(t))-\mathcal{L}_{\epsilon(t)}(\bar{x}^*)\leq\widehat{\mathcal{E}}(t)
	\leq\frac{T^2\widehat{\mathcal{E}}(T)}{t^{2}}, \quad \forall t\geq T.
\end{eqnarray*}
By using Lemma \ref{lemmaztt3.3}, we have
\begin{eqnarray}\label{z29}
	\|x(t)-x_{\epsilon(t)}\|^2+\|x_{\epsilon(t)}\|^2-\|\bar{x}^*\|^2\nonumber\leq\frac{2T^2\widehat{\mathcal{E}}(T)}{t^{2}\epsilon(t)}, \quad \forall t\geq T.
\end{eqnarray}
It follows from \eqref{z26} and $\lim_{t\rightarrow+\infty}t^2\epsilon(t)=+\infty$ that
\begin{eqnarray*}
	\lim_{t\rightarrow+\infty}\|x(t)-\bar{x}^*\|=0.
\end{eqnarray*}

{\bf Case II:} There exists a large enough $T$ such that $\{x(t): t\geq T\}$ stays in $B(0,\|\bar{x}^*\|)$. In this case,
\begin{eqnarray}\label{dudu10}
	\|x(t)\|<\|\bar{x}^*\|, \quad \forall t\geq T.
\end{eqnarray}
Let $\bar{x}\in\mathcal{X}$ be a weak sequential cluster point of $(x(t))_{t\geq t_{0}}$ and let $\{t_n\}_{n\in \mathbb{N}}\subseteq[T, +\infty)$ satisfy $t_n\to+\infty$ and 
\begin{eqnarray}\label{z31}
	x(t_{n})\rightharpoonup\bar{x}   \quad as \quad n\rightarrow +\infty.
\end{eqnarray}
It follows that
\begin{eqnarray*}\label{fyp27-1}
	\mathcal{L}_{\rho}(\bar{x},\bar{\lambda}^*)\leq\liminf_{n\rightarrow+\infty}\mathcal{L}_{\rho}(x(t_n),\bar{\lambda}^*).
\end{eqnarray*}
By using Theorem \ref{theorem:ztt}, we have
\begin{eqnarray*}\label{z30}
	\lim_{n\rightarrow+\infty}\mathcal{L}_{\rho}(x(t_n),\bar{\lambda}^*)-\mathcal{L}_{\rho}(\bar{x}^*,\bar{\lambda}^*)=0.
\end{eqnarray*}
Thus, we can obtain 
\begin{eqnarray*}
	\mathcal{L}_{\rho}(\bar{x},\bar{\lambda}^*)\leq\mathcal{L}_{\rho}(\bar{x}^*,\bar{\lambda}^*).
\end{eqnarray*}
Because of $(\bar{x}^*,\bar{\lambda}^*)\in\Omega$, we have
\begin{eqnarray*}
	\mathcal{L}_{\rho}(\bar{x}^*,\bar{\lambda}^*)=\min_{x\in\mathcal{X}}\mathcal{L}_{\rho}(x,\bar{\lambda}^*)\leq\mathcal{L}_{\rho}(\bar{x},\bar{\lambda}^*))\leq\mathcal{L}_{\rho}(\bar{x}^*,\bar{\lambda}^*).
\end{eqnarray*}
As a consequence,
\begin{eqnarray*}
	\mathcal{L}_{\rho}(\bar{x},\bar{\lambda}^*)=\mathcal{L}_{\rho}(\bar{x}^*,\bar{\lambda}^*)=\min_{x\in\mathcal{X}}\mathcal{L}_{\rho}(x,\bar{\lambda}^*),
\end{eqnarray*}
which implies
\begin{eqnarray*}
	\bar{x}\in\arg\min_{x\in\mathcal{X}}\mathcal{L}_{\rho}(x,\bar{\lambda}^*).
\end{eqnarray*}
It follows from \eqref{ztt: onefang} that  $\bar{x}\in S$. By using \eqref{dudu10}, we have
\begin{eqnarray*}
	\limsup_{n\rightarrow+\infty}\|x(t_{n})\|\leq\|\bar{x}^*\|.
\end{eqnarray*}
On the other hand, since $\|\cdot\|$ is weakly lower semicontinuous, it follows from \eqref{z31} that
\begin{eqnarray*}
	\|\bar{x}\|\leq\liminf_{t_{n}\rightarrow+\infty}\|x(t_{n})\|\leq\limsup_{n\rightarrow+\infty}\|x(t_{n})\|\le\|\bar{x}^*\|,
\end{eqnarray*}
which together with $\bar{x}\in S$ and $\bar{x}^*=Proj_{S}0$ yields $\bar{x}=\bar{x}^*$. Thus, the trajectory $(x(t))_{t\geq t_0}$ has a unique weak sequential cluster point $\bar{x}^*$, which implies
\begin{eqnarray*}
	x(t)\rightharpoonup \bar{x}^*   \quad as \quad t\rightarrow+\infty.
\end{eqnarray*}
Using again \eqref{dudu10}, we obtain
\begin{eqnarray*}
	\limsup_{t\rightarrow+\infty}\|x(t)\|\leq\|\bar{x}^*\|.
\end{eqnarray*}
Since $x(t)\rightharpoonup \bar{x}^*$, we have
$$\|\bar{x}^*\|\leq\liminf_{t\rightarrow+\infty}\|x(t)\|.$$
Combing the last two inequalities, we can obtain
\begin{eqnarray*}
	\lim_{t\rightarrow+\infty}\|x(t)\|=\|\bar{x}^*\|.
\end{eqnarray*}
Hence,
\begin{eqnarray*}
	\lim_{t\rightarrow+\infty}\|x(t)-\bar{x}^*\|=0.
\end{eqnarray*}

{\bf Case III:} For any $T\geq t_0$, the trajectory $\{x(t): t\geq T\}$ does not remain in the ball $B(0,\|\bar{x}^*\|)$, and does not remain in the complement of the ball $B(0,\|\bar{x}^*\|)$, too.  In this case, there exists a sequence $(t_n)_{n\in\mathbb{N}}\subseteq[t_0,+\infty)$ such that $t_n\rightarrow+\infty$ as $n\rightarrow+\infty$, and 
\begin{eqnarray}\label{fyp27-2}
	\|x(t_n)\|=\|\bar{x}^*\|, \quad \forall n\in\mathbb{N}.
\end{eqnarray}
Let $\hat{x}$ be a weak sequential cluster point of $(x(t_n))_{n\in\mathbb{N}}$. By using arguments similar to Case II, we can obtain
$$x(t_n)\rightharpoonup \bar{x}^*, \quad as \quad n\rightarrow+\infty.$$
This together with \eqref{fyp27-2} implies 
\begin{eqnarray*}
	\lim_{n\rightarrow+\infty}\|x(t_n)-\bar{x}^*\|=0.
\end{eqnarray*}
As a consequence,
\begin{eqnarray*}
	\liminf_{t\rightarrow+\infty}\|x(t)-\bar{x}^*\|=0.
\end{eqnarray*}
The proof is complete.	
\end{proof}

\begin{remark}
	Attouch, Chbani, and Riahi proved in \cite[Theorem 4.1]{AttouchZH2018}  that the trajectory $x(t)$ of $\text{(AVD)}_{\alpha, \epsilon}$ converges strongly to the minimal norm solution of the unconstrained optimization problem \eqref{up-f} when the Tikhonov regularization parameter decreases slowly to zero. Following the approach developed in  \cite[Theorem 4.1]{AttouchZH2018}, we prove in Theorem \ref{theoremztt3.1} that  the primal trajectory $x(t)$ of the primal-dual dynamical system \eqref{z2}  enjoys the same convergence property. So,  Theorem \ref{theoremztt3.1}  can be regarded as an extension of \cite[Theorem 4.1]{AttouchZH2018} from the unconstrained case to the optimization problem \eqref{z1}.
\end{remark}

\begin{remark}\label{corollary2.1}
	When $\epsilon(t)=\frac{c}{t^r}$ with $0<r<2$ and $c>0$, we have $\lim_{t\rightarrow+\infty}t^2\epsilon(t)=+\infty$ and $\int_{t_0}^{+\infty}\frac{\epsilon(t)}{t}<+\infty$. Thus, we can obtain all conclusions of Theorem \ref{theoremztt3.1}.
\end{remark}

\section{Numerical experiments}

In this section, we perform some numerical experiments to illustrate our approach associated with the dynamical system \eqref{z2} for solving the linear equality constrained optimization problem \eqref{z1}. Here, we consider two examples: The first  is a linear equality constrained optimization problem with the convex but not strongly convex objective function. The second  is the quadratic programming problem with a linear equality constraint. All codes are run on a PC (with 2.70GHz Dual-Core Intel Core i5 and 4GB memory) under MATLAB Version R2017b.

\subsection{The linear equality constrained optimization problem with the convex but not strongly convex objective function}

Consider the linear equality constrained convex optimization problem
\begin{eqnarray}\label{zyccxztt12}
	\min_{x\in\mathbb{R}^3}  f(x)=(mx_{1}+nx_{2}+ex_{3})^2, \quad\text{ s.t. } Ax=b,
\end{eqnarray}
where $f:\mathbb{R}^3\rightarrow\mathbb{R},$ $A=(m,-n,e)^T$ with $m,n,e\in\mathbb{R}\backslash\{0\}$, and $b=0$. This example is motivated by the example considered in \cite[Section 4]{Laszlo(2023)} for the unconstrained optimization problem. It is easily verified that the solution set is
$\{(x_{1},0,-\frac{m}{e}x_{1}): x_{1}\in\mathbb{R}\}$ and the optimal objective function value is $0$. Obviously, the minimal norm solution of this convex optimization problem is $\bar{x}^*=(0,0,0)$. In our numerical experiments, we take the starting points $x(1)=(1,1,1)$, $\lambda(1)=1$, $\dot{x}(1)=(1,1,1)$ and solve the dynamical system \eqref{z2} numerically with the ode23 in MATLAB.

In the first experiment, we take $m=5,$ $n=1$,  $e=1$, $\alpha=13$, $\epsilon(t)=\frac{3}{t^r}$ and  $\rho=1$. Figure \ref{fig:testfig} shows the behaviors  of  $\|x(t)-\bar{x}^*\|$ and $\mathcal{L}_{\rho}(x(t),\lambda^*)-\mathcal{L}_{\rho}(x^*,\lambda^*)$ along  the trajectory $(x(t))_{t\geq t_{0}}$ generated by our dynamical system \eqref{z2} under the different choices of $r$ with $0<r<2$.
\begin{figure}[htbp]
	\centering
	\begin{minipage}[t]{0.99\linewidth}
		\centering
		\includegraphics[width=2.5in]{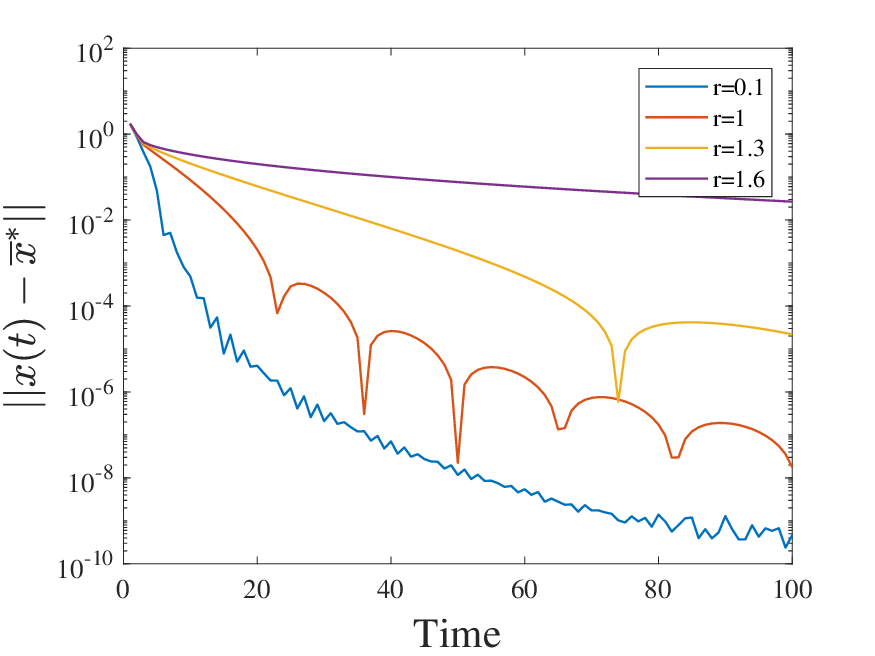}
		\centering
		\includegraphics[width=2.5in]{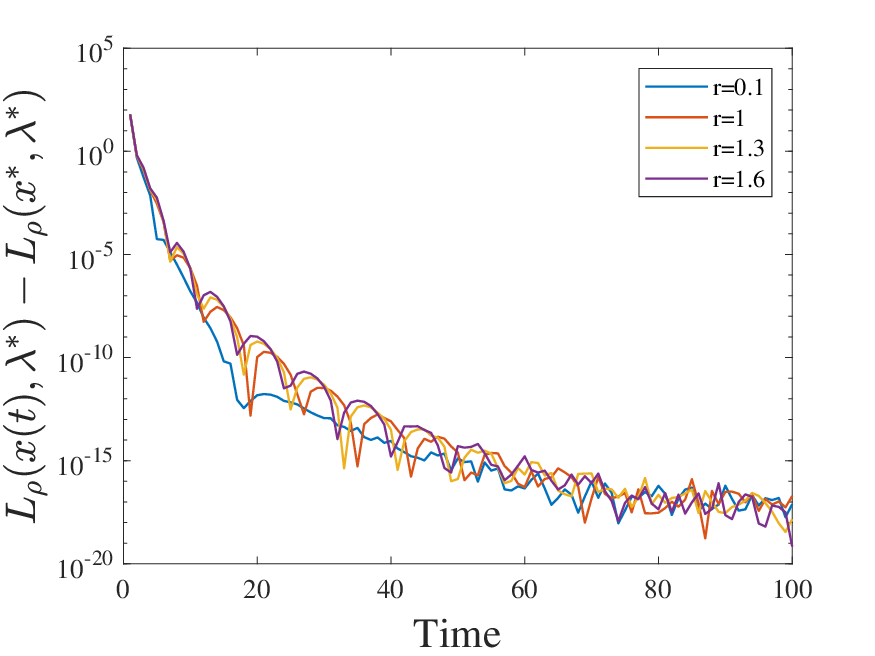}
	\end{minipage}%
	\caption{Error analysis of dynamical system \eqref{z2}  with different Tikhonov regularization parameters for problem \eqref{zyccxztt12}}
	\label{fig:testfig}
\end{figure}

As shown in Figure \ref{fig:testfig}, the trajectory $(x(t))_{t\geq t_{0}}$ generated by \eqref{z2} converges to the minimizer of minimal norm. Further, we can also observe that our dynamical system \eqref{z2} performs better in error $\|x(t)-\bar{x}^*\|$ when the parameter $r$ is small, and that the error $\mathcal{L}_{\rho}(x(t),\lambda^*)-\mathcal{L}_{\rho}(x^*,\lambda^*)$ is not very sensitive to the Tikhonov regularization parameter.

In the second experiment, fix $\alpha=13$, $\rho=1$ and $\epsilon(t)=\frac{3}{t^r}$ with $r=0.5$ in our proposed dynamical system \eqref{z2}, and fix $\alpha=13$, $\beta(t)=1$ and $\varepsilon(t)=0$ in $\text{(He-AVD)}_{\alpha}$. Then, under different choices of $m$, $n$ and $e$, we  illustrate the behaviors of the trajectories $(x(t))_{t\geq t_{0}}$ of Tikhonov regularized dynamical system \eqref{z2} and $\text{(He-AVD)}_{\alpha}$, where our proposed dynamical system \eqref{z2} is a Tikhonov regularized version of $\text{(He-AVD)}_{\alpha}$ with $\beta(t)=1$ and $\varepsilon(t)=0$.
\begin{figure*}[h]
	\centering
	\subfloat[m=5, n=1, e=1.]
	{
		\begin{minipage}[t]{0.31\linewidth}
			\centering
			\includegraphics[width=1.9in]{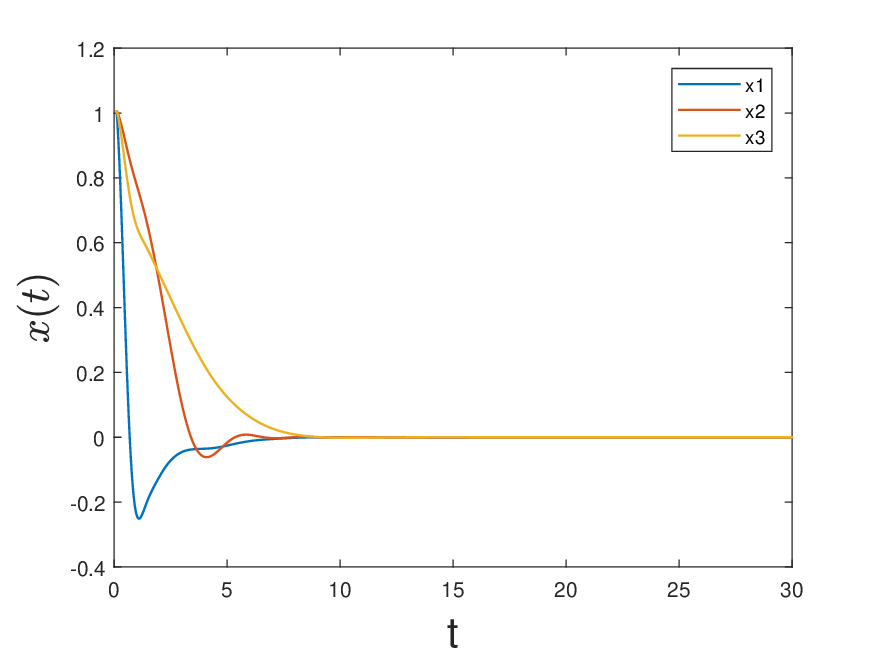}
		\end{minipage}%
	}
	\subfloat[m=50, n=10, e=15.]
	{
		\begin{minipage}[t]{0.31\linewidth}
			\centering
			\includegraphics[width=1.9in]{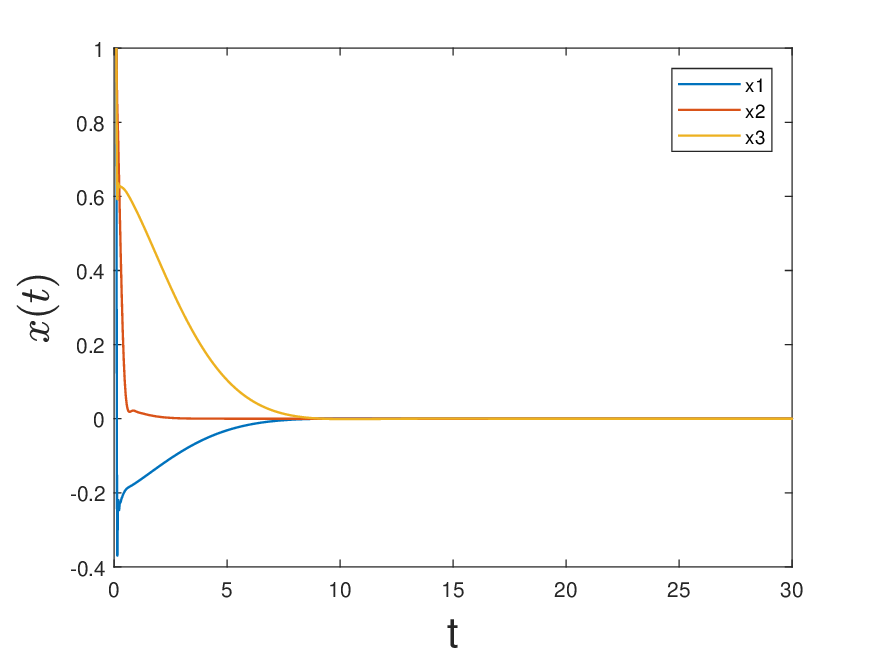}
		\end{minipage}%
	}
	\subfloat[m=200, n=30, e=50.]
	{
		\begin{minipage}[t]{0.31\linewidth}
			\centering
			\includegraphics[width=1.9in]{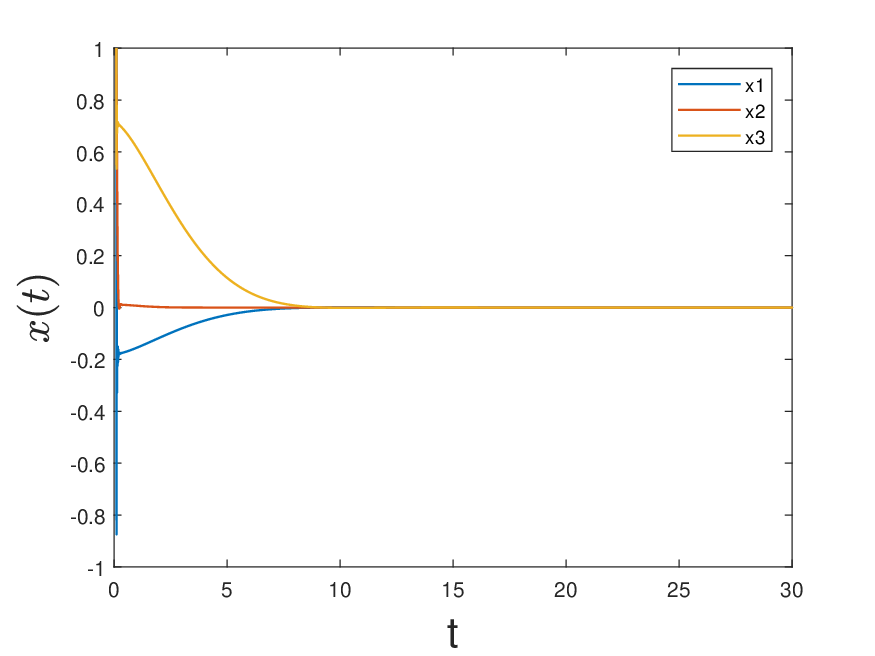}
		\end{minipage}%
	}
	\caption{The behavior of the trajectory $(x(t))_{t\geq t_0}$ generated by dynamical system \eqref{z2} for problem \eqref{zyccxztt12}}
	\label{fig:titstfig}
	\centering
\end{figure*}

\begin{figure*}[h]
	\centering
	\subfloat[m=5, n=1, e=1.]
	{
		\begin{minipage}[t]{0.31\linewidth}
			\centering
			\includegraphics[width=1.9in]{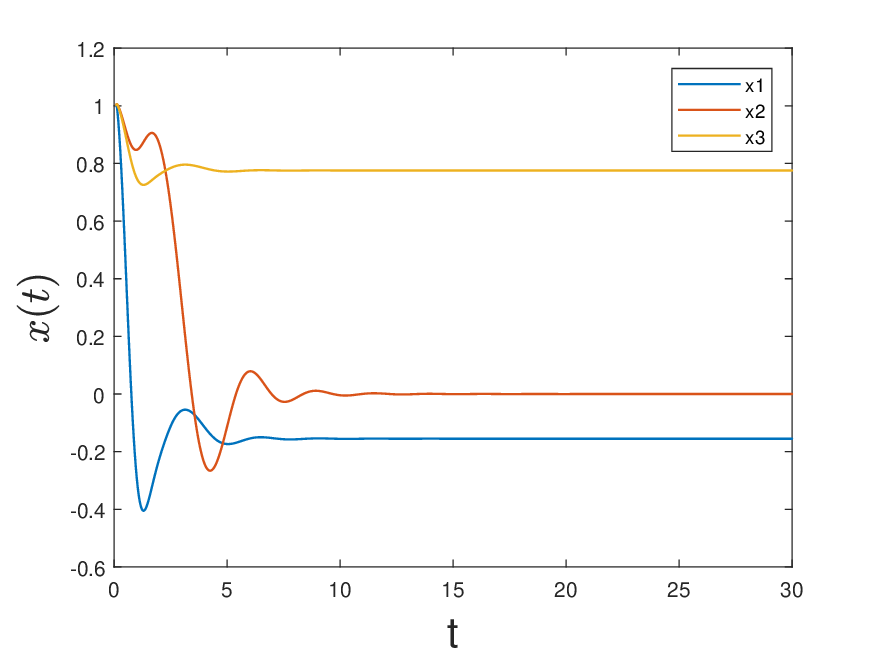}
		\end{minipage}%
	}
	\subfloat[m=50, n=10, e=15.]
	{
		\begin{minipage}[t]{0.31\linewidth}
			\centering
			\includegraphics[width=1.9in]{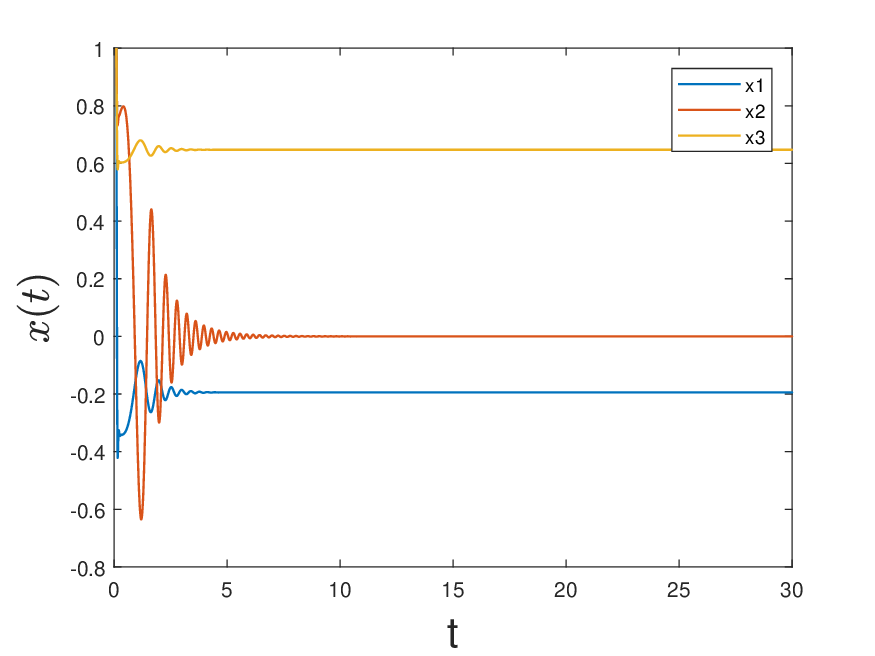}
		\end{minipage}%
	}
	\subfloat[m=200, n=30, e=50.]
	{
		\begin{minipage}[t]{0.31\linewidth}
			\centering
			\includegraphics[width=1.9in]{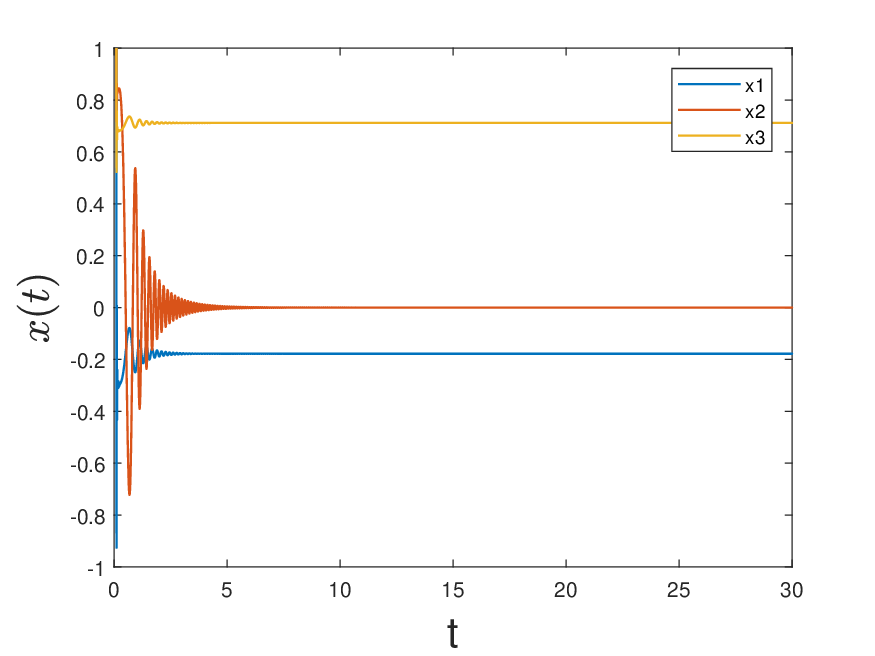}
		\end{minipage}%
	}
	\caption{The behavior of the trajectory $(x(t))_{t\geq t_0}$ generated by $\text{(He-AVD)}_{\alpha}$ for problem \eqref{zyccxztt12}}
	\label{fig:titstcxfig}
	\centering
\end{figure*}
As shown in Figure \ref{fig:titstfig}, the trajectory $(x(t))_{t\geq t_{0}}$ generated by Tikhonov regularized dynamical system \eqref{z2} converges to the minimal norm solution $\bar{x}^*=(0,0,0)$ under different choices of $m$, $n$ and $e$. However, Figure \ref{fig:titstcxfig} shows that the trajectory $(x(t))_{t\geq t_{0}}$ generated by $\text{(He-AVD)}_{\alpha}$ cannot converge to the minimal norm solution $\bar{x}^*=(0,0,0)$. Thus, for the problem \eqref{zyccxztt12}, the numerical results of the second experiment show that Tikhonov regularization term can guarantee that the generated trajectory $(x(t))_{t\geq t_{0}}$ converges to the minimal norm solution.

\subsection{The quadratic programming problem}

In this subsection, we consider the linear equality constrained quadratic programming problem
\begin{equation}\label{LCPP}
	\min_{x\in\mathbb{R}^n} f(x)=\frac{1}{2}x^T Mx+q^T x, \quad \text{ s.t. } Ax=b,
\end{equation}
where $A\in\mathbb{R}^{m\times n}$,  $b\in\mathbb{R}^m$,  $q\in\mathbb{R}^n$, and  $M\in\mathbb{R}^{n\times n}$ is a symmetric and positive semidefinite matrix. 

We test the proposed dynamical method \eqref{z2}  and  make  comparisons with  $\text{(Z-AVD)}_{\alpha, \theta}$ in \cite{ZengXLandLeiJLandChenJ(2022)} and the primal\mbox{-}dual dynamical system with vanishing damping $\text{(PD\mbox{-}AVD)}$ in \cite{BNguyen2022}. We  use the ode23 in MATLAB to solve the dynamical systems numerically on the interval $[1,100]$.

Here, take the starting point $(x(1),\lambda(1),\dot{x}(1))=\mathbf{1}^{(2n+m)\times1}$, $\epsilon(t)=\frac{1}{t^r}$ with $r>2$, $\alpha=15$ and $\rho=1$ in our dynamical system, take the starting point $(x(1),\lambda(1),\dot{x}(1),\dot{\lambda}(1))=\mathbf{1}^{2(n+m)\times1}$, $\theta=\frac{1}{\alpha-1}$, $\alpha=15$ and $\beta=1$ in (PD-AVD) and take the starting point $(x(1),\lambda(1),\dot{x}(1),\dot{\lambda}(1))=\mathbf{1}^{2(n+m)\times1}$, $\theta=\frac{1}{2}$ and $\alpha=15$ in $\text{(Z-AVD)}_{\alpha, \theta}$.  Set $m=20$ and $n=40$. Let $b$ be generated by the uniform distribution. Let  $q$, $A$, and $M=H^T H$ with $H\in\mathbb{R}^{n\times n}$ be generated by the standard Gaussian distribution. We get the optimal value $f(x^*)$  by using the Matlab function $quadprog$ with tolerance $10^{-15}$. The results are depicted in Figure \ref{fig:tgtpstfig}.

\begin{figure}[htbp]
	\centering
	\begin{minipage}[t]{0.99\linewidth}
		\centering
		\includegraphics[width=2.5in]{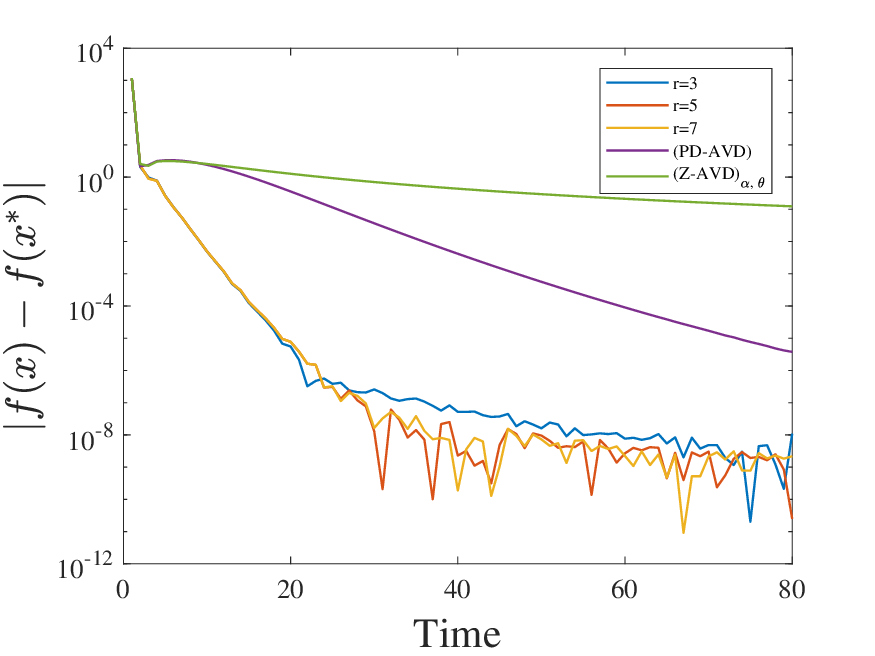}
		\centering
		\includegraphics[width=2.5in]{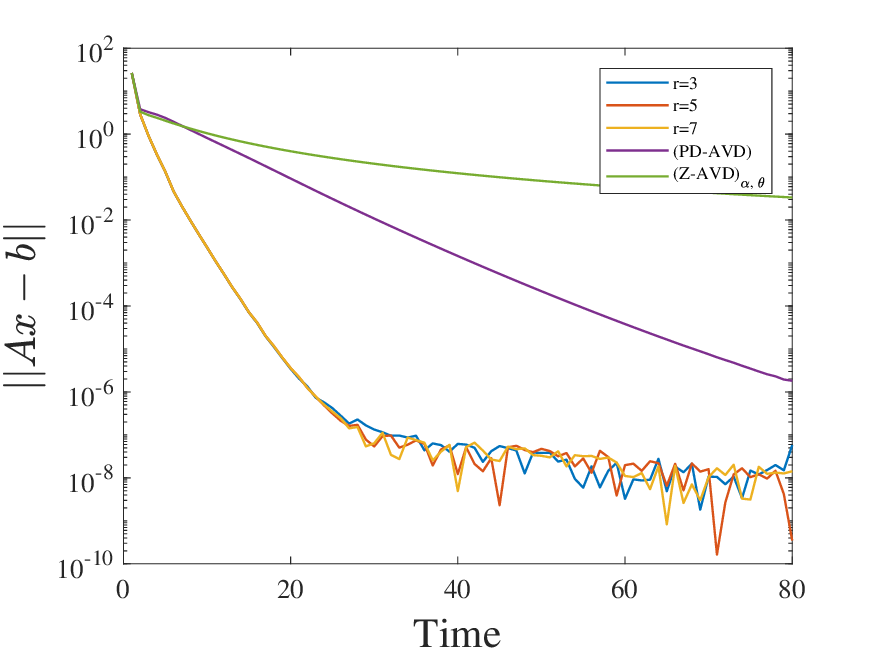}
	\end{minipage}%
	\caption{Comparisons of dynamical system \eqref{z2}, (PD-AVD) and $(Z-AVD)_{\alpha, \theta}$ for problem \eqref{LCPP}}
	\label{fig:tgtpstfig}
\end{figure}

As shown in Figure \ref{fig:tgtpstfig}, our dynamical system \eqref{z2} with a fast vanishing Tikhonov regularization parameter performs better than (PD-AVD) and $\text{(Z-AVD)}_{\alpha, \theta}$ on the quadratic programming problem \eqref{LCPP}. We also observe in Figure \ref{fig:tgtpstfig} that the errors of the objective function and the feasibility violation along the trajectory generated by our dynamical system are not very sensitive to the fast vanishing Tikhonov regularization parameter.

\section*{Disclosure statement}
No potential conflict of interest was reported by the authors.

\end{document}